\documentclass{article}
\usepackage{amsmath, amsthm, amssymb}
\usepackage[all]{xy}
\usepackage[lite]{amsrefs}
\usepackage{comment}
\usepackage{mathrsfs}
\usepackage{researchpreamble}
\usepackage{booktabs}
\usepackage[pagebackref,colorlinks]{hyperref}

\theoremstyle{plain}
\newtheorem*{thm:hyperelliptic-thetachar}{Theorem~\ref{thm:thetachar-Bd}}
\newtheorem*{thm:hyperelliptic-torsion}{Theorem~\ref{thm:torsion-Bd}}

\newcommand{\gk}{\mathfrak{g}}
\newcommand{\Fr}{\sigma}
\newcommand{\ghat}{\chi_{\gamma}}
\newcommand{\tilg}{t_{\gamma}}
\newcommand{\fig}{F_1(\Gamma)(1)}
\newcommand{\fim}{F_1(\gamma,\mu)}

\newcommand{\oc}{\mathscr{C}}
\newcommand{\od}{\mathscr{D}}
\newcommand{\oj}{\mathscr{J}}

\newcommand{\os}{\mathscr{S}}
\renewcommand{\sq}{\mathscr{Q}}
\newcommand{\ku}{K^{u}}

\newcommand{\dr}{\Div^{\{r\}}}
\newcommand{\dro}{\Div^{\{1\}}}
\newcommand{\drz}{\Div^{\{0\}}}

\newcommand{\dt}{\Div^{\{2\}}}
\newcommand{\pt}{\Pic^{\{2\}}}
\newcommand{\pr}{\Pic^{\{r\}}}
\newcommand{\prz}{\Pic^{\{0\}}}

\newcommand{\bh}{\bar{h}}

\newcommand{\lxy}{L_{XY}}
\newcommand{\lzw}{L_{ZW}}
\newcommand{\lmzw}{L_{-ZW}}
\newcommand{\cxy}{C_{XY}}
\newcommand{\czw}{C_{ZW}}
\newcommand{\cmzw}{C_{-ZW}}

\newcommand{\hzw}{h_{ZW}}
\newcommand{\hmzw}{h_{-ZW}}

\begin{document}

\title{A descent map for curves with totally degenerate semi-stable reduction}
\author{Shahed Sharif\thanks{Acknowledgments to Dino Lorenzini and Wayne Aitken for many helpful comments.}
}
\maketitle
% !!! Ideas: do the case where (Ci . Cj) \leq 1 to determine Galois
% action on X(T)
% !!! Deformation problem: totally degenerate curve + canonical divisor

\section{Introduction}
\label{sec:introduction}

Let $K$ be a local field with residue field $k$ of characteristic $p>0$; that is, $K$ is a finite extension of the $p$-adic field $\Q_p$, or it is $\F_q((T))$ where $q$ is a power of $p$. Write $G$ for the absolute Galois group of $K$, and $\gk$ for the absolute Galois group of the residue field. Let $C$ be a smooth, proper, geometrically integral curve over $K$ with genus $g\geq 2$. 
% !!! Perhaps regular is enough? Don't know about models for such curves.
Letting $\oc$ be the minimal proper regular model for $C$ over the discrete valuation ring $\ok$ of $K$, we write $C_k$ for the special fiber. Let $C_{\bar{k}}$ be the base-extension of the special fiber to $\bar{k}$, the algebraic closure of the residue field; it equals the special fiber of the minimal proper regular model for $C$ over $\ku$, the maximal unramified extension of $K$. We will assume that $C$ has totally degenerate semi-stable reduction; that is, $C_{\bar{k}}$ is connected, reduced and consists of a finite collection of $\Pro^1$s such that a formal neighborhood of each singularity is isomorphic to $\Spec \bar{k}[[x,y]]/(xy)$.

%Begin dino's sugg
Let $\Pic_{C/K}$ denote the Picard scheme of $C/K$, and $\Pic C$ the Picard group, with $\Pic C \subseteq \Pic_{C/K}(K)$. (These are equal if $C(K)\neq \emptyset$, but not generally.) Let $r$ be a positive integer. As $\Pic_{C/K}$ is an abelian scheme, we have a multiplication-by-$r$ morphism 
\[
[r]: \Pic_{C/K} \to \Pic_{C/K}
\]
which induces group homomorphisms $[r]:\Pic_{C/K}(K) \to \Pic_{C/K}(K)$ and $[r]:\Pic C \to \Pic C$. Little is known in general about the image of these maps. The canonical sheaf defines a canonical element in $\Pic C$, and it is natural to wonder whether this element is in the image of any of the maps $[r]$. In this article, we investigate this problem in the case $p\nmid r$ (which we henceforth assume), and provide a method for answering this type of question when the reduction of the Jacobian of the curve is purely toric. 
%End dino's sugg.
 Under certain conditions on the $C_k$, we will define a subgroup $\pr C \subset \Pic C$, a finite set of classes $D_i$, a finite collection of finite cyclic groups ${\mu}_i$, and maps ${\gamma}_i:\pr C \to {\mu}_i$ which satisfy the following condition:
\begin{theorem}
  Given $L\in \pr C$, we have $L \in r\Pic C$ if and only if there is some $D_j$ such that $L+D_j \in \ker {\gamma}_i$ for every $i$.
\end{theorem}

The above result, phrased more explicitly, appears as Corollary~\ref{cor:toric-descent-with-functions} below.

The strength of the theorem lies in the explicit description of the $\gamma_i$. Recall that a \emph{theta characteristic} (also called \emph{spin structure}) is an invertible sheaf $L$ whose square $L^{\otimes 2}$ equals the canonical class. As examples we will show, for $C$ lying in certain families of hyperelliptic curves, or in a specific family of genus 4 curves, how to determine if $C$ has a rational theta characteristic, and we will also show how to compute the prime-to-$p$ rational torsion in the Jacobian of $C$.\footnote{In practice, if $K=\Q_p$ then our method often gives the full rational torsion on the Jacobian; see Remark~\ref{rmk:full-torsion}.} Note that there is no known algorithm to determine in general either the reduction type or the size of the rational torsion in the Jacobian of curves of genus $g\geq 3$.

We will prove the following results:
%\begin{theorem}\label{thm:tcs-hyper-descent}
\begin{thm:hyperelliptic-thetachar}
  Let $K$ be a local field with discrete valuation ring $\ok$, uniformizer $\pi$, and residue field $k$ of characteristic $p$. Let $g\in \ok[x]$ be monic of degree $d\geq 3$ and such that $p\nmid 2d$. Let $h\in \ok[x]$ be a polynomial of degree $e$ with $e\leq 2d$. Suppose that $(\pi,g,g')$ and $(\pi, g, h)$ are both the unit ideal in $\ok[x]$. Let $C$ be the nonsingular projective curve with affine piece given by
  \[
  y^2 = g^2 + \pi h.
  \]
  For $x\in \ok$, write $\bar{x}$ for the reduction of $x \pmod{\pi}$. Similarly for $f\in \ok[x]$, write $\bar{f}$ for the reduction in $k[x]$.
  \begin{enumerate}
      \item Suppose $d$ is odd and $\bar{g}$ factors over $k$ as
    \[
    \bar{g}(x) = (x-\alpha_0) g_1(x)\cdots g_s(x)
    \]
    where $\alpha_0\in k$ and each of the $g_i$ is irreducible over $k$. Let $\alpha_i\in \bar{k}$ be a root of $g_i(x)$. Then $C$ has a rational theta characteristic if and only if $\Nm \bar{h}(\alpha_0)\bar{h}(\alpha_i) \in k^{\times 2}$ for all $i$; here $\bar{h}(\alpha_i)\neq 0$ for all $i$, and the norm is computed from $k(\alpha_i)$ to $k$.
      \item If $d$ is even or $\bar{g}$ is irreducible, then $C$ has a rational theta characteristic.
  \end{enumerate}
\end{thm:hyperelliptic-thetachar}
%\end{theorem}

Given $C/K$, we construct its Jacobian $J/K$ and the N\'eron model $\oj$ over $\ok$. Recall that if $C$ has totally degenerate semi-stable reduction, then the special fiber of $\oj$ is an extension of a finite group by a torus $T$.
\begin{thm:hyperelliptic-torsion}
    Let $C$ be as in the previous theorem, but now suppose $\bar{g}$ splits into linear factors over $k$. Suppose $q=\# k$. Let $J(K)(p')$ be the largest subgroup of the $K$-rational torsion on the Jacobian of $C$ which has order prime to $p$. Let $\alpha_i$ be the roots of $\bar{g}$. Let $H$ be the subgroup of $k^\times$ generated by the numbers $\bar{h}(\alpha_i)/\bar{h}(\alpha_0)$. Let $n$ be the order of $\frac{H\cdot k^{\times d}}{k^{\times d}}$ and let $m = d/n$. Then
  \[
  J(K)(p') \isom \bigg(\frac{\Z}{(q-1)\Z}\bigg)^{d-2} \oplus \frac{\Z}{n(q-1)\Z} \oplus \frac{\Z}{m\Z}.
  \]
\end{thm:hyperelliptic-torsion}

We prove stronger versions of the above two results when $d=3$; see Theorems~\ref{thm:thetachar-B3} and~\ref{thm:torsion-B3}.

In \S~\ref{sec:genus-4-nonhyp}, there are similar results for the family of nonhyperelliptic genus 4 curves given in $\Pro^3_K$ as the intersection of the quadric $XY=ZW$ and the cubic $(X-Y)(Z-W)(Z+W) = \pi \eps$, where $\eps$ varies in a Zariski open subset of the set of all homogeneous cubic forms in $\ok[X,Y,Z,W]$. In Theorem~\ref{thm:genus4-torsion}, we calculate the prime-to-$p$ $K$-rational torsion on the Jacobian of such a curve, and in Theorems~\ref{thm:genus4-thetachar} and~\ref{thm:genus4-cube-root} we determine whether there exists a rational theta characteristic and a rational \emph{cube root} of the canonical class. (The literature sometimes refers to $r$th roots of the canonical class as \emph{$r$-spin structures}.)

The methods described below so far only work when the normalizations of the components of $C_k$ are all isomorphic over $k$ to $\Pro^1_k$. In particular, the \emph{index} of such $C/K$, the gcd of degrees of $K$-rational divisors on $C$, is 1. The idea is simple: suppose we wish to determine if $L\in \Pic C$ is $r$-divisible. We translate $L$ by a rational divisor which is known to be $r$-divisible, and such that the class of the translate represents a point on the toric part of the Jacobian. We then apply the theory of algebraic tori. There are a number of technical difficulties to resolve along the way, one of which involves computation of the $K$-rational prime-to-$p$ torsion on the Jacobian---see Propositions~\ref{prop:computing-nu-with-div-f} and Corollary~\ref{cor:torsion-nonsplit}.

Results for determining rationality of theta characteristics, but with other methods, were proven by a number of different authors. Atiyah~\cite{atiyah1971} showed that if the Galois action on the 2-torsion of the Jacobian factors through a cyclic group, then $C$ has a rational theta characteristic; when $\car k \neq 2$ and $C$ has good reduction, this immediately implies the existence of a rational theta characteristic, and motivates our study of degenerating curves. Mumford~\cite{mumford1971} in the case of hyperelliptic curves gave explicit representations of the theta characteristics in terms of the Weierstrass points of $C$; thus, knowing the Galois action on the Weierstrass points enables one to determine if there is a rational theta characteristic.  Parimala and Scharlau~\cite{parimala-scharlau1994} in results extended by Suresh~\cite{suresh1994} found a condition for the rationality of theta characteristics of hyperelliptic curves involving the splitting of a particular quaternion algebra. Suresh also gives a method for computing the order of the 2-torsion subgroup of the Jacobian.

Given a curve $C$ as above, a related question is whether all of the $r$-torsion on the Jacobian of $C$, or more generally all of the $r$th roots of a given line bundle on $C$, are rational over $\ku$. This is essentially a geometric question, and has been answered by Chiodo~\cite{chiodo2009}. Pacini gives a more explicit answer for the case of theta characteristics~\cite{pacini2010}. See also Gross-Harris~\cite{gross-harris2004} for a description of the rational points on the moduli space of curves all of whose theta characteristics are rational.

Finally, see Poonen-Rains~\cite{poonen-rains2011} for a relationship of the rationality of theta characteristics to a certain cup product over an abelian variety. 

\section{Basic facts on algebraic tori over finite fields}
\label{sec:basic-facts-algebr}

In the following, we let $q=\#k$, let $T$ be a $g$-dimensional algebraic torus defined over $k$, and let $X(T)$ be the character group of $T_{\bar{k}}$. Let $\Fr$ be the Frobenius automorphism acting on $\bar{k}$ with fixed field $k$. The character group $X(T)$ is a free $\Z$-module of rank $g$ equipped with an action of $\gk$. In this section we will compute the group of rational points $T(k)$ based on knowledge of $X(T)$ as a $\Z[\gk]$-module.

Henceforth, unless otherwise stated, all modules are $\Z[\gk]$-modules. For example, ``$X(T)$ is generated by $\chi$'' means generated over $\Z[\gk]$.\begin{definition}
  Let $X$ be a $\Z[\gk]$-module. We say $X$ is \emph{principal} if it can be generated by a single element. We say $X$ is \emph{principally decomposable} if it can be written as a direct sum of principal submodules.

  If $X=X(T)$, the character group of an algebraic torus, we use the above terms to describe $X$ and $T$ interchangeably. We say that $T = \prod T_j$ is a \emph{principal decomposition} of $T$ if $T$ is the fiber product over $k$ of the $T_j$, and each $T_j$ is a principal torus.
\end{definition}

\begin{example}
  Let $\ell/k$ be the unique degree $g$ extension. Let $T$ be the Weil restriction of scalars $R_{\ell/k} \G_m$. Recall that $T$ has the universal property that for any $k$-scheme $S$, we have a functorial isomorphism
  \[
  T(S) = \G_m(S_\ell)
  \]
  where $S_\ell$ means the base-extension $S \times_{\text{Spec } k} \text{Spec } \ell$. Explicitly, $T_\ell \isom (\G_{m,\ell})^g$, and the usual action of $\Fr$ is twisted by the automorphism which cyclically permutes the factors; that is,
  \[
  \Fr(x_1,\dots,x_g) = (\Fr x_g, \Fr x_1, \dots, \Fr x_{g-1}).
  \]
Alternatively, we can characterize $T$ by setting $X(T) = \Z[\chi]/(\chi^g-1)$, and $\Fr$ acts as multiplication by $\chi$. We observe that $T$ is principal.
\end{example}

\begin{remark}
  Suppose $T=\prod T_i$ is a principal decomposition of $T$. The projection $T \to T_i$ induces an inclusion $X(T_i) \subset X(T)$. This inclusion will be used without comment from here on.

  From now on, we will assume that unless otherwise stated all tori under discussion are principally decomposable.
\end{remark}

\begin{definition}
  If $T$ is the restriction of scalars $R_{\ell/k}\G_m$ for some finite extension $\ell/k$, then we say that $T$ is a \emph{norm torus}, or \emph{$\ell$-norm torus} for clarity. If $T$ is a product of norm tori, we say it is a \emph{normal torus}.
\end{definition}

\begin{proposition}[Ono~\cite{ono1961}, Prop.~1.2.2]\label{prop:double-dual-torus}
  $T(k) = \Hom_\gk(X(T), \bar{k}^\times)$.
\end{proposition}
The map is given by evaluation; that is, for $x\in T(k)$, we associate the homomorphism $e_x$ given by $e_x(\chi) = \chi(x)$.

 If the action of $\gk$ is trivial, we see that $T(k) = \Hom(X(T), k^\times)$. If the action of $\gk$ factors through $\Gal(\ell/k)$, since $T(k) \hra T(\ell)$, we observe that we may replace $\bar{k}^\times$ in the proposition with $\ell^\times$.

Given a $\Z$-basis of $X(T)$, the action of $\Fr$ may be represented by an element of $\Gl(g, \Z)$. We let $f(x)$ be the characteristic polynomial of the matrix obtained this way; it is independent of the choice of basis.

Given a positive integer $n$ for which $(n,q) = 1$, let $\mu_n$ denote the \'etale sheaf over $k$ of $n$th roots of unity.
\begin{proposition}\label{prop:indecomposable-torus-mu}
  Suppose $T$ is principal. Then $T(k) \isom \mu_{f(q)}(\bar{k}^\times)$.
\end{proposition}
Since $f(q) \equiv\pm 1 \pmod{q}$, $\mu_{f(q)}$ is an \'etale sheaf over $k$.

We give some simple examples to demonstrate the proposition. If $T=\G_m$, then $f(x) = x-1$ and $f(q) = q-1$. The right hand side above becomes $\mu_{q-1}(\bar{k}^\times) = k^\times$.

Now suppose that $T$ is an $\ell$-norm torus, where $[\ell:k]=g$. Our description of $X(T)$ as a $\Z[\gk]$-module shows that $f(x) = x^g - 1$. Then $T(k) \isom \mu_{q^g - 1}(\bar{k}^\times) = \ell^\times$. We verify this by the universal property of Weil restriction of scalars: $T(k) = \G_m(\ell) = \ell^\times$.

\begin{proof}
  Since $T$ is principal, there is some $\chi \in X(T)$ such that $\chi$ is a $\Z[\gk]$-generator for $X(T)$. We will show that the map
  \begin{align*}
    T(k) &\to \mu_{f(q)}(\bar{k}^\times) \\
    x &\mapsto \chi(x)
  \end{align*}
  is an isomorphism.

  First observe that $\chi$, $\chi^\Fr$, \dots, $\chi^{\Fr^{g-1}}$ form a $\Z$-basis for $X(T)$. Thus, the values of these characters uniquely characterize points in $T(\bar{k})$. Given $x\in T(k)$, by Proposition~\ref{prop:double-dual-torus} we must have 
  \[
  \chi^{\Fr^i}(x) = \Fr^i \chi(x) = \chi(x)^{q^i}
  \]
 for every $i$. Therefore the value $\chi(x)$ determines all of the values $\chi^{\Fr^i}(x)$ for $i = 1, \dots, g-1$. This tells us that $\chi: T(k) \to \bar{k}^\times$ is injective. 

Now choose $\omega \in \mu_{f(q)}(\kb^\times)$. We will construct a $\gk$-equivariant homomorphism $e_x:X(T) \to \bar{k}^\times$ for which $e_x(\chi) = \omega$. In order for $e_x$ to be $\gk$-equivariant, we must have
\[
e_x(\chi^{\Fr^i}) = e_x(\chi)^{q^i}
\]
for all $i$. For $0\leq i \leq g-1$, these are independent constraints. The only additional constraint is given by the characteristic polynomial; that is, we have the identity of characters
  \[
  \chi^{f(\sigma)} = 0,
  \]
  or $e_x(\chi)^{f(q)} = 1$. This holds by our choice of $\omega$, and so the proposition follows.
\end{proof}

\begin{theorem}\label{thm:T-sum-mu-i}
  Suppose $T = \prod T_i$ is a principal decomposition of $T$ over $k$. Let $f_i(x)$ be the characteristic polynomial of Frobenius acting on $X(T_i)$. Let $\chi_i$ be a $\Z[\gk]$-generator for $X(T_i)$. Then the map
  \[
  \oplus \chi_i: T(k) \to \oplus \mu_{f_i(q)}(\bar{k}^\times)
  \]
  is an isomorphism.
\end{theorem}

\begin{proof}
  This is an immediate consequence of the previous proposition.
\end{proof}

We will often write $\mu(T_i)$ in place of $\mu_{f_i(q)}(\kb^\times)$. Note that $f_i(q)$, and hence $\mu(T_i)$, depend on the base field.

\begin{corollary}\label{cor:order-torus}
  Let $T$ be principally decomposable torus over $k$, and let $f(x)$ be the characteristic polynomial of Frobenius acting on $X(T)$. Then $\# T(k) = f(q)$.
\end{corollary}

The corollary in fact holds for general algebraic tori; see Ono~\cite[eq.~(1.2.6)]{ono1961}.

\section{The descent map}
\label{sec:the-descent-map}

Recall that $C$ is a smooth, proper, geometrically integral curve over $K$ with minimal proper model $\oc$ having totally degenerate semi-stable reduction and special fiber $C_k$. In this situation, the Jacobian $J$ of $C_k$ has a N\'eron model $\oj$ whose special fiber $J_k$ lies in a short exact sequence of group schemes
\[
0 \to T \to J_k \to \Phi \to 0,
\]
where $T$ is an algebraic torus and $\Phi$ is a finite \'etale group scheme. The scheme $\Phi$ is the \emph{component group} of $J_k$.

\subsection{Overview}
\label{sec:overview}

% [For some reason, one needn't worry about singular support; I guess because of regularity? I'm getting this from Baker's treatment. He might get it by using Cartier divisors. But I start over the generic fiber.]
Starting in \S~\ref{sec:eval-divis-1}, we will assume that the normalization of every irreducible component of $C_k$ is isomorphic to $\Pro^1_k$. This implies that $\Pic C = (\Pic \Cbar)^G$; that is, any divisor linearly equivalent to its Galois conjugates is linearly equivalent to a rational divisor. In the remainder of this overview, we will operate under this assumption.

Let $\dro C$ be the group of divisors $D$ on $C$ supported away from points with singular reduction. Note that for any $L \in \Pic C$, there exists some $D \in \dro C$ which represents it by~\cite[Theorem~3.3]{gabber-liu-lorenzini2011}. Let 
\[
\tau: \dro C \to \Div C_k
\]
be the specialization map; that is, given $D \in \dro C$, let $\od$ be the Zariski closure of $D$ in $\oc$ under the canonical inclusion $C \hra \oc$. Then $\tau(D)$ is $\od \cap C_k$. (See for example~\cite[\S 2.1]{baker2008} or~\cite[ch.~10.1.3]{liu2002}.) If $C_i$ is an irreducible component of $C_k$, the intersection pairing $(\tau(D) \cdot C_i)$ is well-defined; by abuse of notation, we will also write it $(D\cdot C_i)$. 

Given $L \in \Pic C$, we wish to determine if $L$ lies in $r\Pic C$. Evidently it is necessary that $L$ lie in $\Pic^r C$. We will define a group $\pr C$ lying in a filtration
\begin{equation}\label{eq:prc-filtration}
\Pic^r C \supset \pr C \supset r\Pic C.
\end{equation}
In Proposition~\ref{prop:geometric-obstruction-divisibility} below, we show how to determine if $L \in\Pic^r C$ lies in $\pr C$. Then, in the following sections, we will show how to determine if $L \in \pr C$ lies in $r\Pic C$. The two quotients defined by the filtration may be viewed as giving, respectively, geometric and arithmetic obstructions to $r$-divisibility.

Our first step is to define a degree map
\[
\deg: \Div C_k \to \Z^v
\]
where $v$ is the number of irreducible components of $C_k$. Now letting $C_1, \dots, C_v$ be the irreducible components of $C_k$, the degree map is defined by
\[
\deg(D) = ((D \cdot C_i))
\]
where $\cdot$ denotes the intersection pairing. Then $\pr C$ is defined to be the set of divisor classes containing a divisor $D\in\dro C$ such that $\deg(\tau(D)) \in r\Z^v$. (For $D \in \dro C$, we will usually write $\deg(D)$ in place of $\deg(\tau(D))$.) Furthermore, we certainly have the filtration~\eqref{eq:prc-filtration}.

We now show how to determine if a given $L \in \Pic^r C$ lies in $\pr C$. Let $M_{fib}$ be the subgroup of $\Z^v$ generated by the vectors $v_i = ((C_i\cdot C_j))$.

\begin{proposition}\label{prop:geometric-obstruction-divisibility}
  Suppose $L \in \Pic^r C$. Let $D\in \dro C$ represent $L$. Then $L\in \pr C$ if and only if $\deg(D)$ lies in $r\Z^v + M_{fib}$.
\end{proposition}

\begin{remark}\label{rmk:picr-geometric}
  Observe that the equivalent conditions of the theorem are geometric, hence we may base-extend to $K^u$ to prove the theorem.
\end{remark}

\begin{proof}
 Certainly if $L \in \pr C$, we may find a divisor $E$ with class $L$ such that $\deg(E) \in r\Z^v$. Since $D$ is linearly equivalent to $E$, if we consider $D$ and $E$ as horizontal divisors on $\oc$ via the inclusion $C \hookrightarrow \oc$, then $(D-E)$ is linearly equivalent to a fibral divisor $F$. Therefore $\deg(D) = \deg(E) + \deg(F) \in r\Z^v + M_{fib}$.

Suppose now that $\deg(D) \in r\Z^v + M_{fib}$. Then there exists a divisor $D'$ such that $\deg(D-rD') \in M_{fib}$. Replacing $D$ with $D-rD'$, we see that we wish to show that if $\deg D \in M_{fib}$, then $L \in \pr C$.

Let $F \in \Div \oc$ be a fibral divisor for which $\deg F = \deg D$. We wish to replace $F$ with a linearly equivalent horizontal divisor. On each irreducible component $C_i$ of $C_k$, choose a smooth closed point $x_i \in C_i(\kb)$. By~\cite[Prop.~6.1]{gabber-liu-lorenzini2011}, there is a divisor $F'$ linearly equivalent to $F$ which avoids each of the $x_i$; this latter condition forces $F'$ to be horizontal. Observe that $\deg (F') = \deg (F)$. Furthermore, if we consider $F'$ as a divisor on $C$ by restricting to the generic fiber of $\oc$, one sees that $F'$ must be principal. Thus $D$ is linearly equivalent to $D - F'$. But $\deg(D-F') = \deg(D-F) = 0$, from which the claim follows.
\end{proof}

Therefore to determine if $L\in \Pic C$ is divisible by $r$, we may assume that $L\in \pr C$. In this section, we will associate to every character of $T$ a homomorphism $\gamma$ from $\displaystyle\frac{\pr C}{r\Pic C}$ to a finite cyclic group, so that the intersection of the kernels of the $\gamma$s is trivial. The idea behind the homomorphism is as follows.

Given $L \in \pr C$, choose $D\in \dr C$ representing it. We may translate $D$ to $\drz C$ via an $r$-divisible divisor, map the resulting divisor to some $x\in T(k)$, then evaluate characters of $T$ on $x$ to determine if it lies in $rT(k)$. If it does, then $L$ was $r$-divisible. But even if $x\notin rT(k)$, $L$ may still be $r$-divisible in $J(k)$! Under our hypotheses, $\Phi$ is a constant group scheme (see Prop.~\ref{prop:bosch-liu-phi-constant}). We have a short exact sequence
\[
\Phi[r] \to \frac{T(k)}{rT(k)} \to \frac{J(k)}{rJ(k)}.
\]
In order for $L$ to be $r$-divisible, we must test if the resulting $x\in T(k)$ lies in $rT(k) + \im (\Phi[r])$. As we will see, determining $\im (\Phi[r])$ will have the pleasant consequence of telling us the prime-to-$p$ torsion in $J(K)$.

\paragraph{Theta characteristics.}
\label{sec:theta-char}

We will be particularly interested in the case where $L$ is the canonical class and $r=2$; that is, we wish to know when $C$ has a rational theta characteristic. When $\car K = 2$, Mumford~\cite{mumford1971} showed that $C$ \emph{always} has a rational theta characteristic, regardless of the reduction type of the minimal proper regular model for $C$. When $p\geq 3$, work of Atiyah~\cite{atiyah1971} shows that $C$ has a rational theta characteristic over $\ku$. In particular, under our hypotheses on the reduction of $C$, we have that the canonical class must lie in $\pt C$ in all cases except possibly when $\car K = 0$ and $\car k = 2$.

\subsection{One-cycles on $\Gamma$}
\label{sec:one-cycles-gamma}

Let $\Gamma$ be the dual graph of $C_k$; that is, the graph whose vertices correspond to the irreducible components of $C_k$, and whose edges correspond to the nodes of $C_k$. The graph $\Gamma$ comes equipped with a natural $\gk$-action.

Let $H_1(\Gamma,\Z)$ be the group of closed, oriented 1-cycles on $\Gamma$; it is a $\gk$-module. Define a group $F_1(\Gamma)$ as the set of pairs $(\gamma, (t_i))$ where $\gamma \in H_1(\Gamma, \Z)$ and $(t_i)$ is a collection of functions, each $t_i$ being a function on the irreducible component $C_i \subset C_k$ such that the zeroes and poles of $t_i$ occur only at nodes. (If any $C_i$ is itself a nodal curve, construct $t_i$ on its normalization.) The group structure is 
\[
(\gamma_1, (s_i)) \cdot (\gamma_2, (t_i)) = (\gamma_1 + \gamma_2, (s_i\cdot t_i)).
\]
The group $F_1(\Gamma)$ comes equipped with a natural $\gk$-action; namely, 
\[
\Fr(\gamma, (t_i)) = (\Fr \gamma, (t_i^\Fr))
\]
where $t_i^\Fr(x) = \Fr t_i(\Fr^{-1} x)$. There is a canonical Galois-equivariant projection
\[
F_1(\Gamma) \to H_1(\Gamma, \Z).
\]

We define a subset $F_1(\Gamma)(1) \subset F_1(\Gamma)$ as follows. Let $\gamma \in H_1(\Gamma,\Z)$ be an oriented 1-cycle. On $C_{\bar{k}}$, $\gamma$ corresponds to a chain of components $C_0, \dots, C_n$ connected by nodes $x_j$, so that $x_j \in C_j \cap C_{j+1}$. The orientation is given by the order of the components as given; observe that there may be some repetition in the $C_j$ and/or $x_j$. Then $F_1(\Gamma)(1)$ consists of pairs $(\gamma, (t_i))$, where $t_i$ is given as follows:
\begin{enumerate}
    \item if $C_i$ does not appear in $\gamma$, then $t_i=1$;
    \item if $C_i$ appears with multiplicity one, then $t_i$ is a degree 1 local parameter on $C_i$ such that $t_i(x_{i-1}) = 0$ and $t_i(x_i)=\infty$; and
    \item if $C_i$ appears with higher multiplicity, compute a degree 1 local parameter as above for each occurrence of $C_i$, then let $t_i$ be their product.
\end{enumerate}

\subsection{Evaluation of divisors on 1-cycles}
\label{sec:eval-divis-1}

In this section, we will assume that the normalization of every irreducible component of $C_k$ is isomorphic to $\Pro^1_k$.

%%%%%%% Copied to \ref{sec:overview}
% Let $\dro C$ be the group of divisors $D$ on $C$ supported away from points with singular reduction. Let 
% \[
% \tau: \dro C \to \Div C_k
% \]
% be the specialization map; that is, let $\od$ be the Zariski closure of $D$ in $\oc$ under the canonical inclusion $C \hra \oc$. Then $\tau(D)$ is $\od \cap C_k$. (See for example~\cite[\S 2.1]{baker2008} or~\cite[Ch.~10.1.3]{liu2002}.) If $C_i$ is an irreducible component of $C_k$, the intersection pairing $(\tau(D) \cdot C_i)$ is well-defined; by abuse of notation, we will also write it $(D\cdot C_i)$. 
%%%%%% End copied section

Recall from \S~\ref{sec:overview} that $\dro C$ is the group of divisors $D$ on $C$ supported away from points with singular reduction, and that
\[
\tau: \dro C \to \Div C_k
\]
is the specialization map.

Let $\gamma \in \Hp^1(\Gamma,\Z)$ be an oriented 1-cycle, and let $\tilg$ be an element of $F_1(\Gamma)(1)$ lying over $\gamma$. We now define a homomorphism
\[
\tilg:\dro C \to \bar{k}^\times
\]
by
\[
\tilg(D) := \prod t_i(\tau(D) \cap C_i)
\]
where the product is over all $t_i$ appearing in $\tilg$. By abuse of notation, we will often write $D\cap C_i$ for $\tau(D) \cap C_i$, or more concisely just $D_i$. %If we choose $\tilg \in F_1(\Gamma)^{\gk}$, then $\tilg(D) \in k^{\times}$.

Recall the degree map from \S~\ref{sec:overview}, and that $\dr C \subset \dro C$ is the subgroup of divisors $D$ with $\deg D \in r\Z^v$. %\footnote{The notation $\Div^{(r)} C$ usually means that the intersection numbers are \emph{exactly} $r$, so we use braces to denote that the intersection numbers may be any multiple of $r$.}
Let $\pr C$ be the set of divisor classes $L$ such that $L$ contains some divisor in $\dr C$.

\begin{lemma}\label{lemma:H1-isomorphism-character-group}
  For $\gamma \in H_1(\Gamma, \Z)$, choose some $\tilg$ lying over it and consider the homomorphism $\tilg:\drz C_{\ku}\to\bar{k}^\times$. Then the latter homomorphism factors through $\prz C_{\bar{k}}$. Via the identification of $\prz C_{\bar{k}}$ with $T(\bar{k})$, the map $\gamma \mapsto \tilg$ induces a well-defined Galois-equivariant isomorphism
  \[
  H_1(\Gamma, \Z) \to X(T)
  \]
  where $X(T)$ is the character group of the torus $T$.
\end{lemma}

\begin{proof}
  First, observe that if $\tilg=(\gamma, (t_i))$, $\tilg'=(\gamma, (t'_i))$ are two lifts of $\gamma$ to $\fig$, that $t'_i = \alpha_i t_i$ for some $\alpha_i \in \bar{k}^\times$. For $L\in\drz C$, since $(L\cdot C_i) = 0$, we have $t_i(L\cap C_i) = t'_i(L\cap C_i)$. Thus the homomorphism $\drz C_{\ku} \to \bar{k}^\times$ is independent of the choice of $\tilg$.
  
  Now base-extend to $\ku$. We recall the definition of the character $\chi_\gamma$ associated to $\gamma$. Suppose $\gamma$ consists of the components $C_0, \dots, C_{n-1}$, where $C_i$ is connected to $C_{i+1}$ via a node $x_i$ (viewing subscripts modulo $n$). The orientation of $\gamma$ is given by ordering the components with increasing subscripts. Choose $\tilg \in F_1(\Gamma)(1)$ lying over $\gamma$, and let $t_i$ be the corresponding functions. Given $L \in \drz C$ supported away from the nodes, let $L \cap C_i = \sum_j e_{ij} y_{ij}$, where $e_{ij} \in \Z$ and $y_{ij} \in C_i(\bar{k})$. Observe that $\sum_j e_{ij} = 0$ for every $i$. Let 
  \[
  f_i = \prod_j (t_i - t_i(y_{ij}))^{e_{ij}}
  \]
  be a function on $C_i$; we see that $f$ is regular at $x_{i-1}$ and $x_i$. Also,
  \begin{align*}
    f_i(x_{i-1}) &= \prod_j t_i(y_{ij})^{e_{ij}} \\
    f_i(x_i) &= 1
  \end{align*}
  where the first equality uses the fact that $\sum_j e_{ij} = 0$. Then 
  \begin{align*}
    \chi_{\gamma}([L]) &= \prod_i \frac{f_{i+1}(x_i)}{f_i(x_i)} \\
    &= \prod_{ij} t_i(y_{ij})^{e_{ij}} \\
    &= \tilg(L).
  \end{align*}
  The lemma follows.
\end{proof}

Suppose $T=\prod T_i$ is a principal decomposition of $T$, and let $\chi$ be a generator for $X(T_i)$. Suppose that $\gamma \in H_1(\Gamma,\Z)$ corresponds, via Lemma~\ref{lemma:H1-isomorphism-character-group}, to $\chi$. We define a subset $\fim \subset \fig$ as follows. Let $\gamma$ consist of components $C_0, C_1, \dots$ such that $C_i$ and $C_{i+1}$ are connected by the node $x_i$. Since $\Pro^1_k(k)$ has at least 3 elements, we can always find a point $b_i \in C_i(k)$ such that $b_i \neq x_{i-1}, x_{i}$. (Of course, $b_i$ might itself be a node.) Choose $b_i$ for each $C_i$ in the list of components in $\gamma$; note that if a component appears more than once, then the corresponding base points $b_i$ may be different. Then $\fim$ consists of the set of $\tilg = (\gamma, (t_i)) \in \fig$ such that $t_i(b_i) \in \mu(T_i)$. (If $C_i$ appears with multiplicity higher than 1, then this condition must hold for each local parameter which is a factor of $t_i$.) If $T$ is a split torus and no $C_i$ appears with higher multiplicity, this condition is equivalent to $t_i \in k(C_i)^\times$ for all $i$.

\begin{lemma}\label{lemma:tilg-indep-base-pt}
  Let $T=\prod T_i$ be a principal decomposition, and let $\gamma \in H_1(\Gamma,\Z)$ be such that $\chi_\gamma$ is a generator for $X(T_i)$. Suppose that $\tilg = (\gamma, (t_i)), \tilg' = (\gamma, (t'_i)) \in \fim$ lie over $\gamma$. Then there exist $\alpha_i \in \mu(T_i)$ such that $t'_i = \alpha_i t_i$ for every $i$.
\end{lemma}
One consequence of the lemma is that $\fim$ does not depend on the choice of base-points $b_i$.

\begin{proof}
  It suffices to consider the case where $C_i$ appears in $\gamma$ with multiplicity 1. Let $b_i, b'_i \in C_i(k)$ be the base-points corresponding to $t_i, t'_i$, respectively. Then 
\[
t_i(b'_i) = t_i(b_i) \frac{t_i(b'_i)}{t_i(b_i)} = t_i(b_i) \tilg((b'_i) - (b_i)),
\]
where $(b'_i) - (b_i)$ means the natural divisor on $C_k$. But this divisor lies in the reduction of $\drz C$, hence by Lemma~\ref{lemma:H1-isomorphism-character-group} and Theorem~\ref{thm:T-sum-mu-i}, $\tilg((b'_i) - (b_i)) \in \mu(T_j)$. The claim follows from setting $\alpha_i = \tilg((b'_i)-(b_i))$.
\end{proof}

\subsection{Torsion and descent}
\label{sec:torsion-descent}

Given a principal decomposition $T=\prod T_i$, and $\chi$ a generator for $X(T_i)$, let $\gamma \in H_1(\Gamma, \Z)$ be a one-cycle such that, via the isomorphism of Lemma~\ref{lemma:H1-isomorphism-character-group}, $\chi_\gamma = \chi$. Write $r\mu(T_i)$ for the group of all $\alpha^r$, $\alpha \in \mu(T_i)$. We define a map
\[
\dr C \to \frac{\mu(T_i)}{r\mu(T_i)}
\]
which, by abuse of notation, we also denote $\gamma$. This map is defined by
\[
\gamma(D) = \tilg(D) \pmod{r\mu(T_i)}
\]
where $\tilg$ is any element of $\fim$ lying over $\gamma$.

\begin{lemma}
  $\gamma$ is a well-defined homomorphism.
\end{lemma}

\begin{proof}
  For convenience, we assume that $T$ is itself principal, so that $T=T_i$. We need to show that the image of the map $\gamma$ lies in $\mu(T)/r\mu(T)$, and that the map does not depend on the choice of $\tilg$. Lemma~\ref{lemma:tilg-indep-base-pt} implies that $\tilg(D) \in \mu(T)$ for any $\tilg \in \fim, D\in \dr C$. Also by that lemma, choosing a different $\tilg$ is the same as replacing the $t_i$ with $\alpha_i t_i$, where $\alpha_i \in \mu(T)$. But since $D_i$ has degree divisible by $r$, $(\alpha_i t_i)(D_i)$ differs from $t_i(D_i)$ by a power of $\alpha_i^r$, hence an element of $r\mu(T)$.

  Finally, the fact that $\gamma$ is a homomorphism is clear from the definition.
\end{proof}

\begin{lemma}\label{lemma:principal-divr-same}
  Let $f, g \in \Kb(C)^\times$, and suppose $\dv f, \dv g \in \dro C$. If 
  \[
  \deg(\dv f) = \deg(\dv g),
  \]
  then for any $\tilg\in \fig$
  \[
  \tilg(\dv f) = \tilg (\dv g).
  \]
\end{lemma}

\begin{proof}
  It suffices to show that 
  \[
  \tilg\bigg(\dv\bigg(\frac{f}{g}\bigg)\bigg) = 1.
  \]
  This follows from Lemma~\ref{lemma:H1-isomorphism-character-group}.
\end{proof}

Recall that $\Phi$ denotes the component group of $J_k$; it is a finite \'etale group scheme on $\Spec k$. Recall also the exact sequence of group schemes
\[
0 \to T \to J_k \to \Phi \to 0.
\]
(Note that $\Phi$ depends on the base field $K$, and in particular becomes larger upon ramified base change.) The group $\Phi(\kb)$ is effectively computable; see~\cite[Proposition~8.1.2]{raynaud1970} or~\cite[Appendix A]{baker2008}. The rational component group, $\Phi(k)$, can be computed via Theorem~1.11 of~\cite{bosch-liu1999}.

\begin{proposition}\label{prop:bosch-liu-phi-constant}
  \begin{enumerate}
      \item If every component of $C_k$ is geometrically irreducible, then $\Phi$ is a constant group scheme.
      \item For every field extension $\ell/k$, the map $J_k(\ell) \to \Phi(\ell)$ is surjective.
  \end{enumerate}
\end{proposition}
The two statements are respectively Corollary~1.8 and Lemma~2.1b of~\cite{bosch-liu1999}. As the hypothesis of part 1 holds for us, by abuse of notation $\Phi$ will denote $\Phi(\kb)$ as well as its usual meaning.

Consider the commutative diagram
\begin{equation}
  \xymatrix{
    0 \ar[r] & T(k) \ar[r] \ar[d]^r & J_k(k) \ar[r] \ar[d]^r & \Phi \ar[r] \ar[d]^r & 0 \\
    0 \ar[r] & T(k) \ar[r]  & J_k(k) \ar[r]  & \Phi \ar[r]  & 0,}\label{eq:snake-setup-nu}
\end{equation}
  where the vertical maps are multiplication by $r$ and the horizontal sequences are exact by the last proposition. Applying the Snake Lemma, we obtain the exact sequence
  \begin{equation}
    J_k(k)[r] \to \Phi[r] \overset{\nu}{\to} \frac{T(k)}{rT(k)} \to \frac{J_k(k)}{rJ_k(k)}.\label{eq:snake-cbdy-nu}
  \end{equation}

  We now show how to compute $\nu$. 
  \begin{proposition}\label{prop:computing-nu-with-div-f}
    Let $D\in \dro C$ represent $\delta\in\Phi[r]$. Choose any $f\in \Kb(C)^\times$ such that $\deg(\dv f) = -r\deg(D)$. Let $\gamma \in H_1(\Gamma, \Z)$, be any 1-cycle such that the corresponding element of the character group, $\ghat \in X(T)$, generates $X(T_i)$. Then
    \[
     \ghat(\nu(\delta)) \equiv \gamma(\dv f) \pmod{r\mu(T_i)},
    \]
    where by abuse of notation we write $\ghat$ for the induced homomorphism $\frac{T(k)}{rT(k)} \to \frac{\mu(T_i)}{r\mu(T_i)}$.
  \end{proposition}

  It is not hard to show that $f$ may be taken to be the product of defining equations for components of $C_k$.

  \begin{proof}
    For convenience, we will append the subscript $t$ or $b$ to the groups in~\eqref{eq:snake-setup-nu} to distinguish objects in the top row from objects in the bottom row; e.g., $T(k)_t$ refers to the top left object in~\eqref{eq:snake-setup-nu}. Also, for a divisor $D$, write $[D]$ for the divisor class of $D$.

    We proceed with an explicit diagram-chase: start with $\delta \in \Phi[r]$, and consider it as an element of $\Phi_t$. Choose $D \in \dro C$ so that the class $[\tau(D)] \in J_k(k)_t$ maps to $\delta$. Then $[r\tau(D)] \in J_k(k)_b$ maps to $0$ in $\Phi_b$, and so lies in the subgroup $T(k)_b$; the class of this element in $T(k)/rT(k)$ is none other than $\nu(\delta)$.

  Now $[rD] = [rD + \dv f]$. But by our choice of $f$, $rD + \dv f$ is an element of $\Div^{\{0\}} C$. Lemma~\ref{lemma:H1-isomorphism-character-group} tells us that 
  \[
  \ghat(\nu(\delta)) = \tilg(rD + \dv f)
  \]
  for any choice of $\tilg \in \fim$ lying over $\gamma$. The first claim now follows by observing that
  \[
  \tilg(rD + \dv f) \equiv \gamma(\dv f) \pmod{r\mu(T_i)}.
  \]
  \end{proof}

  \begin{corollary}\label{cor:torsion-nonsplit}
    Fix the notation of Prop.~\ref{prop:computing-nu-with-div-f}. If $\tilde{\delta} \in J(k)$ maps to $\delta$, then
    \[
    \nu(\delta) \equiv r\tilde{\delta} \pmod{rT(k)}.
    \]
  \end{corollary}

  \begin{proof}
    One observes that $\tilde{\delta}$ is identical to $[\tau(D)]$ in the proof of Prop.~\ref{prop:computing-nu-with-div-f}. From this observation and the snake lemma argument in the above proof, the claim follows.
  \end{proof}

  \begin{remark}\label{rmk:full-torsion}
    If $M$ is an abelian group, write $M(p')$ for the torsion in $M$ with order prime to $p$. Hensel's Lemma and the $r$-divisibility of the kernel of reduction ($(r,p) = 1$) shows that $J(K)(p') = J(k)(p')$. Thus by varying $r$ and applying the theory of elementary divisors, one can use Proposition~\ref{prop:computing-nu-with-div-f} and Corollary~\ref{cor:torsion-nonsplit} to effectively compute the torsion subgroup $J(K)(p')$. For example, if for all $r$ ($p\nmid r$) we find that $\nu$ is the zero map, then $J(K)(p') \isom T(k) \oplus \Phi(p')$. 

    If $K=\Q_p$ with $p\neq 2$, then the kernel of reduction is torsion-free~\cite[Appendix]{katz1981}; if in addition $\Phi[p]=0$, then we are able to compute the full torsion subgroup of $J(K)$.

    Even when $K \neq \Q_p$ but $\Phi[p] = 0$, one can still compute $J(K)(p)$ by Mumford-Tate uniformization. This will be explored in future work.
  \end{remark}

%%%%%%%%%%%%%%%%

  \begin{theorem}\label{thm:descent-with-nu}
    Let $T=\prod T_j$ be a principal decomposition of $T$. Let $\gamma_j \in H_1(\Gamma, \Z)$ be such that $\chi_{\gamma_j}$ is a $\Z[\gk]$-generator for $X(T_j)$. Given $D \in \dr C$, $[D] \in r\Pic C$ if and only if there exists some $\delta\in\Phi[r]$ such that
    \[
    \gamma_j(L)\in\chi_{\gamma_j}(\nu(\delta)) \cdot r\mu(T_j)
    \]
 for all $j$.
  \end{theorem}

  \begin{proof}
    Let $D' \in \dro C$ satisfy $r(D'\cdot C_i) = (D\cdot C_i)$ for all $i$. Then $D \in r\Pic C$ if and only if $D-rD' \in r\Pic C$. Furthermore 
    \[
    \gamma_j(D) \equiv \gamma_j(D-rD') \pmod{r\mu(T_j)}
    \]
    for all $j$. Therefore we may assume that $D \in \drz C$. In particular, the reduction of the class of $[D]$ to the special fiber of the Jacobian $J_k$ lies in the torus $T(k)$. Let $x \in T(k)$ be the point corresponding to $[D]$. From the exact sequence~\eqref{eq:snake-cbdy-nu}, $[D]$ is divisible by $r$ in $J_k(k)$ if and only if $x$ lies in the subgroup generated by $rT(k)$ and $\im \nu$. The claim follows.
  \end{proof}

  \begin{remark}\label{rmk:div-Cunr}
    The proof, together with Lemma~\ref{lemma:principal-divr-same}, shows that for ${L} \in \Pic C$, we may use any $D \in \dr C_{\ku}$ which represents ${L}$; then $\gamma_i(D) \in \mu(T_i)$, and the above theorem holds.
  \end{remark}

  \begin{corollary}\label{cor:toric-descent-with-functions}
    For each $\delta \in \Phi[r]$, let $D_\delta \in \dro C_{\ku}$ represent it, and choose $f_\delta$ so that $\deg(\dv f_\delta) = -r\deg(D_\delta)$. Given $D \in \dr C_{\ku}$ with $[D] \in \Pic C$, we have $[D] \in r\Pic C$ if and only if there exists some $\delta$ such that $\gamma_i(D+\dv f_\delta) \in r\mu(T_i)$ for every $i$.
  \end{corollary}

  \begin{proof}
    Combine Theorem~\ref{thm:descent-with-nu}, Remark~\ref{rmk:div-Cunr}, and Proposition~\ref{prop:computing-nu-with-div-f}.
  \end{proof}

\subsection{Simplifications for normal tori}
\label{sec:simpl-norm-tori}

We suppose that $T$ is a principal torus with $X(T)$ generated by $\chi$. Let $\gamma \in H_1(\Gamma, \Z)$ correspond to $\chi$. Consider $t_\gamma \in \fim$. Let $m$ be the smallest positive integer such that $\Fr^m \gamma = \gamma$. We define $\Nm t_\gamma$ to be
\[
\Nm t_\gamma = \prod_{i=0}^{m-1} \Fr^i t_\gamma
\]
where the product is computed in $F_1(\Gamma)$. The projection of $\Nm t_\gamma$ lies over the 1-cycle $\Nm \gamma:=\sum_{i=0}^{m-1} \Fr^i \gamma$, which generates $H_1(\Gamma, \Z)^\gk$. Given a torus with a fixed principal decomposition, extend $\Nm$ in the obvious way to generators $\gamma_i$ for each principal subtorus. For $\gamma \in H_1(\Gamma, \Z)^\gk$, we may write it as
\[
\gamma = \sum e_i \Nm \gamma_i
\]
for some integers $e_i$; if $T$ is a normal torus, this representation is unique. We then define $\fim$ to be the set of elements in $F_1(\Gamma)$ of the form
\[
t_\gamma = \prod (\Nm t_{\gamma_i})^{e_i}
\]
where $t_{\gamma_i} \in F_1(\gamma_i, \mu(T_i))$. Note that the map
\[
t_\gamma: \dr C \to k^\times
\]
induces a well-defined map
\[
\gamma: \pr C \to \frac{k^\times}{k^{\times r}}.
\]

  \begin{corollary}\label{cor:descent-normal-torus}
    Suppose $T$ is a normal torus and $r\mid (q - 1)$. Given $D \in \dr C_{\ku}$ with $[D] \in \Pic C$, we have $[D] \in r\Pic C$ if and only if there exists some $\delta\in\Phi[r]$ such that
    \[
    \gamma(D)\in\chi_{\gamma}(\nu(\delta)) \cdot k^{\times r}
    \]
 for all $\gamma \in H_1(\Gamma, \Z)^{\gk}$.
  \end{corollary}

  \begin{proof}
    It suffices to consider the case where $T$ is a norm torus; say $X(T) \isom \Z[\chi]/(\chi^g -1)$, with Frobenius acting as multiplication by $\chi$. Let $\ell$ be the unique degree $g$ extension of $k$. Then $\chi:T(k) \to \ell^\times$ is an isomorphism. Furthermore, $X(T)^\gk$ is generated by
    \[
    \Nm \chi:=1+\chi+\chi^\Fr + \cdots + \chi^{\Fr^{g-1}}.
    \]
    Observe that the diagram
    \[
    \xymatrix{
      T(k) \ar[d]^\chi \ar[rd]^{\Nm \chi} & \\
      \ell^\times \ar[r]^{\Nm} & k^\times
      }
      \]
      commutes. Now we tensor everything in the diagram with $\Z/r\Z$ and observe that since $r\mid(q-1)$, the induced norm map
      \[
      \Nm: \frac{\ell^\times}{\ell^{\times r}} \to \frac{k^\times}{k^{\times r}}
      \]
      is an isomorphism. The claim follows.      
  \end{proof}

% There should be a fix for the below for normal tori.
  % \begin{corollary}\label{cor:torsion-kstar-normaltorus}
  %   Let $T$ be a normal torus. Suppose $\Phi(p') = \oplus \Z/r_i\Z$, and let $d_i$ be a generator for the $i$th cyclic factor. Setting $r:=r_i$, write $\nu$ for the induced map on $\Phi[r_i]$ and let $H_i \subset \frac{k^\times}{k^{\times r_i}}$ be the subgroup of all $\chi(\nu(d_i))$, $\chi \in X(T)^\gk$. If $\mathrm{lcm}(r_i) \mid (q - 1)$, then
  %   \[
  %   J(K)(p') \isom T(k) \oplus_i \frac{k^\times}{H_i \cdot k^{\times r_i}}.
  %   \]
  % \end{corollary}

  % \begin{proof}
  %   The proof is just a combination of the proofs of Proposition~\ref{prop:torsion-kstar-splittorus} and Corollary~\ref{cor:descent-normal-torus}.
  % \end{proof}

\section{Curves with $\Gamma = B_d$}
\label{sec:curves-with-gamma}

 The graph $B_d$ is the so-called \emph{banana graph}; it consists of 2 vertices connected to each other by $d$ edges:
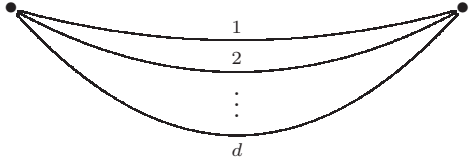
\[
\xygraph{
  !{<0cm,0cm>;<1cm,0cm>:<0cm,1.3cm>::}
  !{(0,1) }*{\bullet}="t"
  !{(6,1) }*{\bullet}="b"
  !{(3,.1) }*{\vdots}
  "t"-@/_{1pc}/"b"^1
  "t"-@/_{2pc}/"b"^2
  "t"-@/_{4pc}/"b"_d
}
\]
 Equivalently, the special fiber of a curve $C$ having dual graph $B_d$ consists of two $\Pro^1$s which intersect transversely at $d$ points; the involution swapping the two components shows that such a curve must be hyperelliptic. For certain families of curves with dual graph $B_d$, we determine rationality of theta characteristics and then compute the order of the prime-to-$p$ rational torsion on the Jacobian.

\subsection{Rationality of theta characteristics when $\Gamma = B_d$}
\label{sec:rati-theta-char-bd}

We apply Theorem~\ref{thm:descent-with-nu} in this section to determine if certain curves with dual graph $B_d$ have rational theta characteristics. We first deal with the case $d=3$ for clarity, then repeat the computation for higher $d$ under weaker hypotheses. Recall that a theta characteristic is a square root of the canonical class. 

\paragraph{The case $d=3$.}
\label{sec:case-d=3}

Suppose we are given a genus 2 curve $C$ over a field $K$ with affine piece
\[
y^2 = g(x)^2 + \pi h(x)
\]
which satisfies the following hypotheses:
\begin{itemize}
    \item $K$ is a local field with residue characteristic $p\geq 5$, and uniformizer $\pi$;
    \item $g(x)\in \ok[x]$ is a monic cubic polynomial such that $(\pi, g, g')\subset \ok[x]$ is the unit ideal (so the reduction of $g(x)$ modulo $\pi$ is separable); and
    \item $h(x)\in \ok[x]$ is a polynomial of degree $e\leq 6$ such that $(\pi, g, h)\subset \ok[x]$ is the unit ideal.
\end{itemize}
We will say that any such $C$ satisfies hypothesis $(H)$. As we shall see, such a curve has dual graph $B_3$.
\begin{lemma}\label{lemma:hyp-H-regular}
  Let $C$ satisfy hypothesis $(H)$, and let $\oc$ be the associated arithmetic surface over $\ok$; i.e., use the same equation and adjoin the usual two points at infinity.\footnote{That is, use the same equation to obtain $\oc^0\subset \mathbb{A}^2_{\ok}$, let $\tilde{\oc}$ be the projective closure, and let $\oc$ be the desingularization at infinity of $\tilde{\oc}$.} Then $\oc$ is minimal and regular, the special fiber has dual graph $B_3$, and the component group $\Phi$ of the Jacobian of $C$ over $K$ is $\Z/3\Z$.
\end{lemma}

\begin{proof}
  We first show regularity. One checks that the special fiber of $\oc$ at infinity is smooth. Then for the given affine piece, we obtain the equation $y^2 = g^2$, which looks like two rational curves intersecting 3 times, once at each root of $g$. We need only check that $\oc$ is regular at these nodes. But $(\pi, g, h) = \ok[x]$ implies that the reduction of $h$ modulo $\pi$ does not vanish at any of the nodes; regularity follows. One also observes that the dual graph is $B_3$.

  Minimality follows from checking Castelnuovo's criterion.

  For the component group, we follow the standard method (see~\cite{raynaud1970} or~\cite{baker2008}). Since every component of $C_k$ possesses a $k$-rational point, the component group is unchanged upon base-extension to the maximal unramified extension $K^u$. One checks that the intersection matrix $M$ for $C_k$ is
  \[
  M =
  \left[
    \begin{array}{rr}
      -3 & 3 \\
      3 & -3
      \end{array}\right].
    \]
    Then $\Phi$ is the homology of $\Z^2 \stackrel{M}{\to} \Z^2 \to \Z$, where the first map is multiplication by $M$ and the second is $(x,y) \mapsto x+y$. We obtain $\Phi \isom \Z/3\Z$.
\end{proof}

Let $\alpha_0, \alpha_1, \alpha_2$ be the roots of $\bar{g}(x)$, where the $\alpha_i$ lie in $\bar{k}$.

  \begin{theorem}\label{thm:thetachar-B3}
    Let $C$ be a curve satisfying hypothesis $(H)$.
    \begin{enumerate}
        \item If $\bar{g}(x)$ splits over $k$, then $C$ has a rational theta characteristic if and only if ${\bar{h}(\alpha_0)\bar{h}(\alpha_1)}$ and $\bar{h}(\alpha_0)\bar{h}(\alpha_2)$ lie in $k^{\times 2}$.
        \item If $\bar{g}(x)$ has a single root over $k$---say, $\alpha_0$---then $C$ has a rational theta characteristic if and only if ${\bar{h}(\alpha_1)\bar{h}(\alpha_2)}$ lies in $k^{\times 2}$.
        \item If $\bar{g}(x)$ is irreducible, then $C$ has a rational theta characteristic.
    \end{enumerate}
  \end{theorem}
Our hypotheses on $g$ and the characteristic of $k$, together with Hensel's Lemma, imply that any factorization of $\bar{g}(x)$ over $k$ lifts to one for $g(x)$ over $K$.
  
  \begin{proof}
    The special fiber $C_k$ has 2 components $C^+$ and $C^-$ meeting transversely at 3 points, given by the $\alpha_i$. Label the components so that the point at infinity on $C$ given by $\frac{y}{x^3} = 1$ lies on $C^+$; we call this point $\infty^+$, and the other point at infinity $\infty^-$. Observe that $C^+$ is given by $y =g(x), \pi = 0$ and $C^-$ by $y=-g(x), \pi = 0$. That means the coordinate $x$ may be used to specify points on each of $C^+$ and $C^-$. Given $\alpha \in k\cup \{\infty\}$, we write $\alpha^+$ for the point $P \in C^+$ such that $x(P) = \alpha$, and similarly $\alpha^-$ for the point $P\in C^-$ such that $x(P) = \alpha$; this is consistent with our labeling of $\infty^{\pm}$. Also, $\alpha^+ =\alpha^-$ if and only if $g(\alpha) = 0$.

  We have $\Gamma$ is the banana graph $B_3$, and the set of edges is isomorphic as a $\gk$-set to the set of roots of $\bar{g}$.

  By Lemma~\ref{lemma:hyp-H-regular}, $\Phi[2] = 0$ and $\im \nu = 0$.

  In order to apply Theorem~\ref{thm:descent-with-nu}, we must find $L \in \dt C$ which is a canonical divisor. The usual choice, $\infty^+ + \infty^-$, meets each component in degree 1. Additionally, every Weierstrass point reduces to a node.

  Let $\tilde{D} = \dv (y-g)$ and $D$ the specialization of $\tilde{D}$ to $C_k$; that is, $D = \tau(\tilde{D})$, where $\tau$ is the specialization map described at the beginning of \S~\ref{sec:eval-divis-1}.

  To compute $D$ we first consider the divisor of $y-g$ as a function on $\oc$. Write $\dv_\oc (y-g) = \tilde{D} + D_f$, where $\tilde{D}$ is horizontal and agrees with the generic divisor $\dv_C (y-g)$, and $D_f$ is fibral. As observed earlier, $D_f = C^+$. Since $\tilde{D} + D_f$ is principal, we have
  \begin{align*}
  (\tilde{D} \cdot C^{+}) &= -(D_f \cdot C^+) = 3 \\
  (\tilde{D} \cdot C^{-}) &= -(D_f \cdot C^-) = -3.
  \end{align*}
  Thus we may write $D = D^+ - D^-$, where $D^{\pm}$ is a degree $3$ divisor supported on $C^{\pm}$, respectively.

  To compute the $x$-coordinates of the points in the support of $D$, we solve
  \[
  g^2 = y^2
  \]
  which yields
  \[
  h(x) \equiv 0 \pmod{\pi}.
  \]
  Recall that $\deg h = e$. Let $\beta_1, \dots, \beta_e$ be the roots of $h$ taken modulo $\pi$, possibly with repetition. Since $g$ and $h$ are relatively prime, none of these roots is equal to any of the $\alpha_i$. Furthermore, the points $(\beta_i, \bar{g}(\beta_i))$ lie on $C^+$. Therefore
  \[
  D^+ = \sum \beta_i^+ + (3-e) \infty^+
  \]
  and so
  \[
  D = \sum \beta_i^+ + (3-e) \infty^+ - 3 \infty^-.
  \]
  Then 
  \[
  L := D+\infty^++\infty^- = \sum \beta_i^+ + (4-e) \infty^+ - 2 \infty^-
  \]
  is the specialization of a canonical divisor. Furthermore, $(L\cdot C^+) = 4$, $(L\cdot C^-) = -2$, and so $L \in \dt C$.

  Suppose we are in the first case: $\bar{g}(x)$ splits. Then $T$ is a split torus and we may apply Corollary~\ref{cor:descent-normal-torus}. In this case, $H_1(\Gamma,\Z)$ is a trivial Galois-module, so we must choose a $\Z$-basis for it. Let $\gamma_1$ be the loop given by 
  
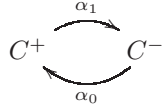
\[
  \xymatrix{
    C^+ \ar@/^1pc/[r]^{\alpha_1} & C^- \ar@/^1pc/[l]^{\alpha_0}
  }
  \]
  Let $\gamma_2$ be the loop given by replacing $\alpha_1$ with $\alpha_2$ above. We construct functions $t_i$ on each component; we call these instead $t_\pm$ in accordance with the labeling on the components of $C_k$. They may be chosen to be
  
  % \begin{tabular}[c]{cc}
  %   $\gamma_1$ & $\gamma_2$ \\ \hline
  %   $t_+=\frac{x-\alpha_0}{x-\alpha_1}$ & $t_+=\frac{x-\alpha_0}{x-\alpha_2}$\\
  %   $t_-=\frac{x-\alpha_1}{x-\alpha_0}$ & $t_-=\frac{x-\alpha_2}{x-\alpha_0}$.
  % \end{tabular}
  
  \begin{align*}
    \gamma_1
    \begin{cases}
      t_+=\frac{x-\alpha_0}{x-\alpha_1}\\
      t_-=\frac{x-\alpha_1}{x-\alpha_0}      
    \end{cases} \qquad
    \gamma_2
    \begin{cases}
      t_+=\frac{x-\alpha_0}{x-\alpha_2}\\
      t_-=\frac{x-\alpha_2}{x-\alpha_0}.
    \end{cases}
  \end{align*}
  For convenience, we consider only the case of $\gamma_1$. Clearly 
  \[
  t_\pm(\infty^-) = t_\pm(\infty^+) = 1.
  \]
  It remains to evaluate
  \[
  \prod_i (\beta_i - \alpha_j)
  \]
  for $j=0,1$. Since the $\beta_i$ are the roots of $\bar{h}(x)$, $\prod (\beta_i-\alpha_j)$ is the constant term of $-\bar{h}(x+\alpha_j)$ divided by the lead coefficient of $\bar{h}$, or just $-\bar{h}(\alpha_j)/h_e$, where $h_e$ is the lead coefficient. Therefore
  \begin{align*}
   \gamma_1(L) &\equiv \bigg(-\frac{{\bar{h}}(\alpha_0)}{h_e}\bigg)\bigg(-\frac{h_e}{{\bar{h}}(\alpha_1)}\bigg) \\
   &\equiv {\bar{h}(\alpha_0)\bar{h}(\alpha_1)} \pmod{k^{\times 2}}.
  \end{align*}
  A similar calculation shows that $\gamma_2(L)\equiv {\bar{h}(\alpha_0)\bar{h}(\alpha_2)}$. The theorem follows for the first case.

  In the second case, $T$ is a principal norm torus. Observe that $H_1(\Gamma, \Z)^{\gk}$ is generated by $\gamma:=\gamma_1 + \gamma_2$. One sees that $\gamma(L) \equiv \gamma_1(L) \cdot \gamma_2(L) \equiv \bar{h}(\alpha_1)\bar{h}(\alpha_2) \pmod{k^{\times 2}}$; the conclusion now follows from Corollary~\ref{cor:descent-normal-torus}.

  In the last case, the action of Frobenius on $X(T)$ with respect to the basis $\gamma_1$, $\gamma_2$ is
\[
\left[
  \begin{array}{cc}
    0 & -1 \\
    1 & -1
    \end{array}\right].
  \]
  According to Corollary~\ref{cor:order-torus}, $\# T(k) = q^2 + q + 1$. In particular, $T(k)$ has odd order, and so $T(k)/2T(k) = 0$. The conclusion follows from Theorem~\ref{thm:descent-with-nu}.
\end{proof}

\begin{example}
  If $C$ is given by
  \[
  y^2 = (x^3-x)^2 + \pi,
  \]
  then $h\equiv 1$, and $C$ has a rational theta characteristic. The Weierstrass points are all rational over $K(\sqrt{\pi})$, so the Galois action on $J[2]$ factors through a cyclic group of order $2$. Atiyah~\cite{atiyah1971} showed that if the Galois action on $J[2]$ factors through a cyclic group, then $C$ has a rational theta characteristic; this verifies our result.

  On the other hand, if $C$ is given by 
  \[
  y^2 = (x^3-x)^2 + \pi x + 2\pi,
  \]
  then $h(x) = x+2$, $\alpha_i = 0, \pm 1$, and $h(0)h(1)=6$, $h(0)h(-1)=2$. Therefore given $p\geq 5$, the curve $C$ has a rational theta characteristic over $\Q_p$ if and only if $p\equiv \pm 1 \pmod{24}$.
\end{example}

\paragraph{The case of general $d$.}
\label{sec:case-dgeq-3}

Fix an integer $d\geq 3$.  Suppose that $C$ is given by $y^2 = g^2 + \pi h$ subject to the following:
\begin{itemize}
    \item $K$ is a local field with uniformizer $\pi$ and residue characteristic $p$, with $p\nmid 2d$;
    \item $g(x)\in\ok[x]$ is a monic polynomial of degree $d$ such that $(\pi, g, g')\subset \ok[x]$ is the unit ideal; and
    \item $h(x)\in\ok[x]$ is a polynomial of degree $e\leq 2d$ such that $(\pi, g, h)\subset \ok[x]$ is the unit ideal.
\end{itemize}
We say such a curve satisfies hypothesis $(H_d)$.
\begin{lemma}\label{lemma:hd-regular}
  Let $C$ satisfy hypothesis $(H_d)$, and let $\oc$ be the associated arithmetic surface over $\ok$; i.e., use the same equation and adjoin the usual two points at infinity. Then $\oc$ is minimal and regular, the special fiber has dual graph $B_d$, and the component group $\Phi$ of the Jacobian of $C$ over $K$ is $\Z/d\Z$.
\end{lemma}

\begin{proof}
  The reasoning is the same as in Lemma~\ref{lemma:hyp-H-regular}, except in this case the special fiber consists of two $\Pro^1$s intersecting at $d$ nodes; thus, the dual graph $\Gamma$ associated to $C$ is the banana graph $B_d$. Furthermore, the intersection matrix for $C_k$ is
  \[
  \left[
    \begin{array}{rr}
      -d & d\\
      d & -d
      \end{array}\right]
    \]
    from which we deduce $\Phi \isom \Z/d\Z$.
\end{proof}

By the comments at the start of \S~\ref{sec:curves-with-gamma}, the curve $C$ is a genus $d-1$ hyperelliptic curve.
\begin{theorem}\label{thm:thetachar-Bd}
  Let $C$ satisfy hypothesis $(H_d)$.
  \begin{enumerate}
      \item Suppose $d$ is odd and $\bar{g}$ factors over $k$ as
    \[
    \bar{g}(x) = (x-\alpha_0) g_1(x)\cdots g_s(x)
    \]
    where $\alpha_0\in k$ and each of the $g_i$ is irreducible. Let $\alpha_i\in \bar{k}$ be a root of $g_i(x)$. Then $C$ has a rational theta characteristic if and only if $\Nm \bar{h}(\alpha_0)\bar{h}(\alpha_i) \in k^{\times 2}$ for all $i$; here $\bar{h}(\alpha_i)\neq 0$ for all $i$, and the norm is computed from $k(\alpha_i)$ to $k$.
      \item If $d$ is even or $\bar{g}$ is irreducible, then $C$ has a rational theta characteristic.
  \end{enumerate}
\end{theorem}

\begin{proof}
  Suppose we are in the first case. As in the set-up of Theorem~\ref{thm:thetachar-B3}, the equation describes a regular arithmetic scheme whose special fiber consists of 2 components, $C^+$ and $C^-$, intersecting at $d$ nodes given by the roots of $\bar{g}$. Let $\beta_i$ be the roots, counted with multiplicity, of $\bar{h}$ over $\bar{k}$. Then as before,
  \[
  L := \sum \beta_i^+ + (2g-e) \infty^+ - 2 \infty^-
  \]
  is a divisor representing the canonical class, and such that $(L\cdot C^\pm)$ are both divisible by $2$.

  By Lemma~\ref{lemma:hd-regular}, $\Phi[2] = 0$. Thus we need only evaluate $L$ on the relevant 1-cycles.

  We now show that the toric part of the special fiber of the Jacobian is in fact a normal torus. Write $e_i$ for the edge on the dual graph $\Gamma$ of $C$ corresponding to the node at $x=\alpha_i$, oriented from $C^+$ to $C^-$. Let $\gamma_i$ be the 1-cycle given by $e_i - e_0$. Observe that the $\gamma_i$ form a $\Z[\gk]$-basis for $X(T)$. Let $d_i = \deg g_i$. The subtorus $T_i$ corresponding to $\Z[\gk]\cdot \chi_{\gamma_i}$ is a norm torus, as the latter module is isomorphic to $\Z[Y]/(Y^{d_i} - 1)$---the isomorphism is given by $\chi_{\gamma_i} \mapsto Y$, and Frobenius acts as multiplication by $Y$. But $T = \prod T_i$, which proves that $T$ is a normal torus.

  It remains to apply Corollary~\ref{cor:descent-normal-torus}. The group $H_1(\Gamma,\Z)^{\gk}$ is generated by the elements
  \[
  \Nm \gamma_i := (1+\Fr+\Fr^2 + \cdots + \Fr^{d_i-1})\gamma_i.
  \]
  By a similar argument as in Theorem~\ref{thm:thetachar-B3}, we see that 
\[
(\Nm \gamma_i)(L) \equiv \Nm \bar{h}(\alpha_0)\bar{h}(\alpha_i) \pmod{k^{\times 2}},
\]
and the first case of the theorem is proved.

  Let us now consider the case where $\bar{g}(x)$ is irreducible and $d$ is odd. Let $\alpha$ be any root of $\bar{g}(x)$, and let $e$ be the edge on the dual graph corresponding to $\alpha$, oriented in either direction. Then the 1-cycles $\Fr^{i+1} e - \Fr^{i} e$ form a $\Z$-basis for $X(T)$. The characteristic polynomial of Frobenius may then be calculated to be $x^{d-1} + x^{d-2} + \cdots + 1 = 0$. Via Proposition~\ref{prop:indecomposable-torus-mu},
  \[
  \# T(k) = q^{d-1} + q^{d-2} + \cdots +1
  \]
  which is \emph{odd}; therefore $T(k)/2T(k) = 0$, and there is a rational theta characteristic.

  If $d$ is even, the proof is even easier: in this case, $\frac{g-1}{2}(\infty^++\infty^-)$ is already a rational theta characteristic.
\end{proof}

When $d$ is even, $C$ is an odd genus hyperelliptic curve, and the above argument shows that such curves always have a rational theta characteristic, regardless of the shape of the special fiber.

\subsection{Calculating torsion on Jacobians when $\Gamma = B_d$}
\label{sec:calc-tors-jacob}

We now apply Proposition~\ref{prop:computing-nu-with-div-f} and the calculations in the proofs of Theorems~\ref{thm:thetachar-B3} and~\ref{thm:thetachar-Bd} to compute the prime-to-$p$ rational torsion on Jacobians of curves satisfying hypothesis $(H)$ or $(H_d)$.

\begin{theorem}\label{thm:torsion-B3}
  Let $C$ satisfy hypothesis $(H)$. Let $\alpha_0, \alpha_1, \alpha_2$ be the roots of $\bar{g}(x)$ in $\kb$. Suppose the order of $k$ is $q$. Let $J(K)(p')$ be the largest torsion subgroup of the rational points on the Jacobian of $C$ with order coprime to $p$. 
  \begin{enumerate}
      \item Suppose $\bar{g}(x)$ splits over $k$. If $\frac{\bar{h}(\alpha_0)}{\bar{h}(\alpha_1)}, \frac{\bar{h}(\alpha_0)}{\bar{h}(\alpha_2)}$ both lie in $k^{\times 3}$, then
    \[
    J(K)(p') \isom \bigg(\frac{\Z}{(q-1)\Z}\bigg)^2 \oplus \frac{\Z}{3\Z}.
    \]
    Otherwise,
    \[
    J(K)(p') \isom \frac{\Z}{3(q-1)\Z} \oplus \frac{\Z}{(q-1)\Z}.
    \]

      \item Suppose $\bar{g}(x)$ has a single root in $k$, say $\alpha_0$. If $q\equiv 1 \pmod{3}$ and $\frac{\bar{h}(\alpha_0)^2}{\bar{h}(\alpha_1)\bar{h}(\alpha_2)}$ lies in $k^{\times 3}$, or if $q\equiv 2\pmod{3}$ and 
    \[
    \bar{h}{(\alpha_1)}^{\frac{q^2-1}{3}} = 1,
    \]
    then
    \[
    J(K)(p') \isom \frac{\Z}{(q^2-1)\Z} \oplus \frac{\Z}{3\Z}.
    \]
    Otherwise
    \[
    J(K)(p') \isom \frac{\Z}{3(q^2-1)\Z}
    \]

      \item Suppose $\bar{g}(x)$ is irreducible over $k$. If $q\equiv 2 \pmod{3}$, or both $q\equiv 1 \pmod{3}$ and 
    \[
    \bar{h}(\alpha_0)^{\frac{q^3-1}{3}} = 1,
    \]
    then
    \[
    J(K)(p') \isom \frac{\Z}{(q^2+q+1)\Z} \oplus \frac{\Z}{3\Z}.
    \]
    Otherwise
    \[
    J(K)(p') \isom \frac{\Z}{3(q^2+q+1)\Z}.
    \]
  \end{enumerate}
\end{theorem}

\begin{proof}
  Let $\delta$ be a generator for $\Phi$ such that $\delta$ is represented by a divisor $D$ with $(D\cdot C^-) = -(D\cdot C^+) = 1$. In each of the three cases, we will apply Proposition~\ref{prop:computing-nu-with-div-f} with $r=3$ to compute $\nu(\delta)$.  Letting $\gamma_i$ and $\chi_{\gamma_i}$ be as in the proof of Theorem~\ref{thm:thetachar-B3}, we will repeatedly use the fact that
  \[
  \chi_{\gamma_i}(\nu(\delta)) = \gamma_i(\dv (y-g)).
  \]

  In the first case, $T$ is a split torus. We have an isomorphism of groups
  \[
  (\chi_{\gamma_1},\chi_{\gamma_2}): T(k) \stackrel\sim\longrightarrow k^\times \oplus k^\times \isom \frac{\Z}{(q-1)\Z} \oplus \frac{\Z}{(q-1)\Z}.
  \]
  From the proof of part 1 of Theorem~\ref{thm:thetachar-B3}, 
  \begin{align*}
    \chi_{\gamma_1}(\nu(\delta)) &= \gamma_1(\dv (y-{g})) \\ 
    &= \frac{\bar{h}(\alpha_0)}{\bar{h}(\alpha_1)}
  \end{align*}
  and similarly for $\gamma_2$. The conclusion now follows from Proposition~\ref{prop:computing-nu-with-div-f}.

%   In the first case, $T$ is a split torus, so we apply Proposition~\ref{prop:torsion-kstar-splittorus}. Since $T\isom \G_m^2$, we have $T(k) \isom (\Z/(q-1)\Z)^2$. The subgroup $H$ in that Proposition is generated by $\gamma_i(\dv (y-{g}))$ for $i=1,2$, where $\gamma_i$ are as in the proof of Theorem~\ref{thm:thetachar-B3}. From the proof of part 1 of the latter theorem, $\gamma_1(\dv (y-{g})) = \frac{\bar{h}(\alpha_0)}{\bar{h}(\alpha_1)}$, and similarly for $\gamma_2$. Observe that if $3\nmid (q-1)$, then $H\subset k^{\times 3}$ automatically. The conclusion in the first case follows.

  In the second case, $T$ is an $\ell$-norm torus, where $\ell/k$ is the unique quadratic extension. Then $\chi_{\gamma_1}$ yields an isomorphism
  \[
  \chi_{\gamma_1}: T(k) \stackrel\sim\longrightarrow \ell^\times \isom \frac{\Z}{(q^2-1)\Z}
  \]
  Observe that $3\mid (q^2-1)$. By Proposition~\ref{prop:computing-nu-with-div-f}, $\ker \nu \isom \Z/3\Z$ if and only if $\gamma_i(\dv (y-g))$ lies in $\ell^{\times 3}$ for $i=1,2$, and $\ker \nu = 0$ otherwise. We know that $\gamma_i(\dv (y-g)) = \bh(\alpha_0)/\bh(\alpha_i)$. If $3\mid (q-1)$, then the map induced by the norm
  \[
  \frac{\ell^\times}{\ell^{\times 3}} \to \frac{k^\times}{k^{\times 3}}
  \]
  is an isomorphism, and the claim follows. If $3\nmid (q-1)$, then $\bar{h}(\alpha_0) \in k^\times = k^{\times 3}$, and since $\bar{h}(\alpha_i)$ for $i=1,2$ are conjugate over $k$, it suffices to determine whether $\bar{h}(\alpha_1)$ lies in $\ell^{\times 3}$. But $\ell^\times$ is cyclic of order $q^2-1$, so the second case is proved.

  In the third case, $T$ is principal, but not a norm torus. As stated in the proof of Theorem~\ref{thm:thetachar-B3}, the characteristic polynomial of Frobenius acting on $X(T)$ is $f(x) = x^2 + x + 1$, and by Proposition~\ref{prop:indecomposable-torus-mu}, 
  \[
  T(k) \isom \frac{\Z}{(q^2+q+1)\Z}
  \]
  If $q\equiv 2 \pmod{3}$, then $3 \nmid (q^2 + q + 1)$ and $T(k)/3T(k) = 0$. Therefore $\ker \nu = \Z/3\Z$, and we obtain the corresponding conclusion. 

  Now suppose $q\equiv 1 \pmod{3}$. Then $\gamma_1(\dv (y-g)) = \bh(\alpha_0)/\bh(\alpha_1)$; since $\mu(T)$ is cyclic, this lies in $3\mu(T)$ if and only if
  \[
  \bigg(\frac{\bh(\alpha_0)}{\bh(\alpha_1)}\bigg)^{\frac{q^2+q+1}{3}} = 1.
  \]
  Observe that 
  \[
  \frac{\bh(\alpha_0)}{\bh(\alpha_1)}\cdot \Fr \bigg(\frac{\bh(\alpha_0)}{\bh(\alpha_1)}\bigg) = \frac{\bh(\alpha_0)}{\bh(\alpha_2)}
  \]
  and so $\bh(\alpha_0)/\bh(\alpha_1) \in 3\mu(T)$ if and only if $\bh(\alpha_0)/\bh(\alpha_2) \in 3\mu(T)$.

  Without loss of generality, $\alpha_1 = \Fr \alpha_0 = \alpha_0^q$. Then
  \begin{align*}
    \bigg(\frac{\bh(\alpha_0)}{\bh(\alpha_1)}\bigg)^{\frac{q^2+q+1}{3}} &=   \big(\bh(\alpha_0)^{q+1}\big)^{-\frac{q^2+q+1}{3}} \\
    &= \bh(\alpha_0)^{-\frac{q^3-1}{3}}.
  \end{align*}
  This completes the proof.
\end{proof}

\begin{theorem}\label{thm:torsion-Bd}
  Let $C$ satisfy hypothesis $(H_d)$, and suppose that $g(x)$ splits completely. Let $\alpha_i$ be the roots of $\bar{g}$. Let $q=\#k$. Let $H$ be the subgroup of $k^\times$ generated by the numbers $\frac{\bar{h}(\alpha_i)}{\bar{h}(\alpha_0)}$. Let $n$ be the order of $\frac{H\cdot k^{\times d}}{k^{\times d}}$ and let $m = d/n$. Then
  \[
  J(K)(p') \isom \bigg(\frac{\Z}{(q-1)\Z}\bigg)^{d-2} \oplus \frac{\Z}{n(q-1)\Z} \oplus \frac{\Z}{m\Z}.
  \]
\end{theorem}

\begin{proof}
  Since $g$ splits completely, $T$ is a split torus of dimension $d-1$. By Lemma~\ref{lemma:hd-regular}, $\Phi \isom \Z/d\Z$. Then there is a generator $\delta$ for $\Phi$ represented by a divisor $D$ with $(D \cdot C^-) = -(D \cdot C^+) = 1$. Observe that $\deg(\dv (y-g)) = (-d,d)$. Let $\gamma_i$ and $\chi_{\gamma_i}$ be as in the proof of Theorem~\ref{thm:thetachar-Bd}. Observe that we have an isomorphism
  \[
  \oplus \chi_{\gamma_i}: T(k) \stackrel{\sim}{\too} (k^\times)^{d-1} \isom \bigg(\frac{\Z}{(q-1)}\bigg)^{d-1}.
  \]
  Applying Proposition~\ref{prop:computing-nu-with-div-f} with $r=d$, we have that $\chi_{\gamma_i}(\nu(\delta))$ equals $\gamma_i(\dv (y-g))$. Furthermore
  \[
  \gamma_i(\dv (y-g)) = \frac{\bh(\alpha_0)}{\bh(\alpha_i)}.
  \]
  It follows that the order of $\nu(\delta)$ is $n$, and the order of $\ker \nu$ is $m$. The claim now follows from Proposition~\ref{prop:computing-nu-with-div-f} and Remark~\ref{rmk:full-torsion}
\end{proof}

\section{A genus 4 nonhyperelliptic family}
\label{sec:genus-4-nonhyp}

In this section, we obtain similar results as in \S~\ref{sec:curves-with-gamma} on a family of nonhyperelliptic genus 4 curves. As $2g-2 = 6$, one can also speak of \emph{cube roots} of the canonical class (also called 3-spin structures); we determine if any of these are rational as well.

Let the base field $K$ be a local field with discrete valuation ring $\ok$, uniformizer $\pi$, and residue characteristic $p$ with $p\geq 5$. For any ``integral'' element ($x\in \ok$, $f\in \ok[X]$, etc.), we use a bar to denote reduction modulo $\pi$ ($\bar{x}$, $\bar{f}$, etc.). Our family will then be the intersection in $\Pro^3_K$ of the quadric surface $XY=ZW$ and the family of cubics
\[
(X-Y)(Z-W)(Z+W) = \pi \eps
\]
where $\eps$ varies in a subset of integral homogeneous cubic forms.

Given a projective variety $V$ over $K$ defined by integral equations, we will use script $\mathscr{V}$ to denote the model of $V$ over $\ok$ obtained by using the same equations. We also write $V_k$ for the special fiber of $\mathscr{V}$.

The recipe is as follows. We first record some general facts about our curves, including the shape of the special fiber. Then we compute the prime-to-$p$ rational torsion on the Jacobian of our curves. Finally, we will determine if there are any rational square roots and cube roots of the canonical class.

\subsection{General facts}
\label{sec:general-facts-1}

A nonhyperelliptic genus 4 curve may be given as the intersection of an irreducible quadric surface and an irreducible cubic surface in $\Pro^3$; this in fact gives the curve in its canonical embedding (see for example~\cite[IV.5.2.2 and IV.5.5.2]{hartshorne1977}). Our quadric $Q$ will be given in projective coordinates $[X:Y:Z:W]$ by 
\[
Q: XY = ZW.
\]
Let $\sq$ be the corresponding arithmetic scheme in $\Pro^3_{\ok}$; i.e.~we use the same equation. Set
\begin{align*}
  \lxy &= X-Y \\
  \lzw &= Z-W \\
  \lmzw&= Z+W.
\end{align*}
Let $\eps$ be a homogeneous cubic form in $\ok[X,Y,Z,W]$. Let $S$ be the cubic surface given by $\lxy\cdot\lzw\cdot\lmzw = \pi \eps$. Let $\os$ be the corresponding arithmetic scheme in $\Pro^3_{\ok}$. Let $\oc$ be the (scheme-theoretic) intersection $\sq\cap \os$ with generic fiber $C$.
\begin{lemma}\label{lemma:eps-regular}
  If $\bar{\eps} \in k[X,Y,Z,W]$ does not vanish at any of the points
  \begin{gather*}
    [1:1:1:1], [-1:-1:1:1]\\
    [i:i:-1:1], [-i:-i:-1:1]\\
    [1:0:0:0], [0:1:0:0]
  \end{gather*}
  where $i$ denotes any fixed square root of $-1$ in $\bar{k}$, then $\oc$ is a minimal regular arithmetic surface such that the components of the special fiber are geometrically integral, and such that the dual graph of the special fiber is as pictured in Figure~\ref{fig:dual-genus4}.
\end{lemma}

As an example of such an $\eps$ when $p\geq 7$, let
\[
\eps_0 = X^3 + Y^3 + WZ^2.
\]

\begin{proof}%[Proof of Lemma~\ref{lemma:eps-regular}]
We first compute the dual graph. The special fiber of $\oc$ is given by 
\[
XY=ZW, \quad \lxy\cdot\lzw\cdot\lmzw = 0.
\]
Write $\cxy$ for the intersection of $Q_k$ and $\lxy = 0$ on the special fiber; define $\czw$ and $\cmzw$ similarly. Then the special fiber has 3 components: $\cxy, \czw$, and $\cmzw$. These are each type $(1,1)$ divisors on the quadric, and so every pair intersects in two points. Solving for these intersections, we obtain the dual graph in the figure.

\begin{figure}\centering
    \input{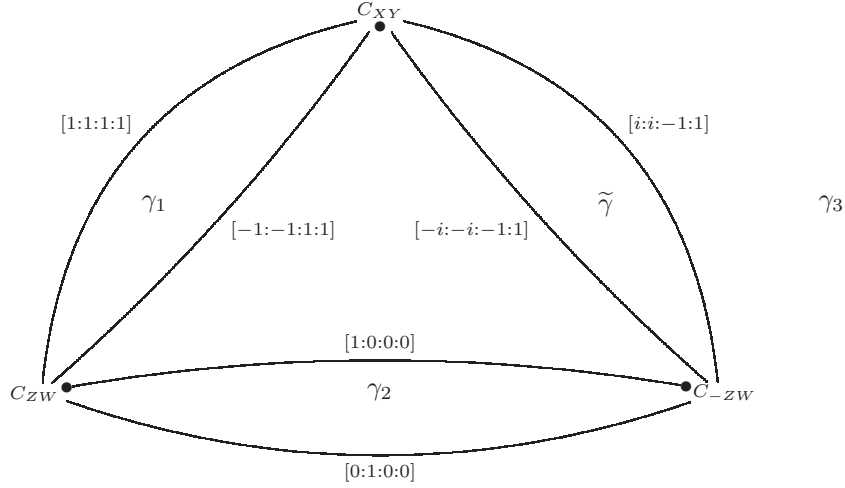}
  \caption{Dual graph of $C_k$}
\label{fig:dual-genus4}
\end{figure}

For regularity, we need only check that $\oc$ is regular at the nodes of the special fiber. By hypothesis, $\bar{\eps}$ does not vanish at any of the nodes, whence the claim follows.

  Minimality follows from checking Castelnuovo's criterion.
\end{proof}

In the figure, four loops are labeled $\gamma_1$, $\gamma_2$, $\gamma_3$ and $\tilde{\gamma}$, where $\gamma_3$ is the loop consisting of the three outside edges; we orient each loop counterclockwise. We define a fifth loop $\gamma_4$ as $\gamma_3 - \tilde{\gamma}$.

Henceforth, we will assume that $\eps$ satisfies the lemma.

The Jacobian $J$ of $C$ is an extension of a torus $T$ by the component group $\Phi$; see Lemma~\ref{lemma:genus4-component-group} for the computation of $\Phi$. The torus $T$ is a normal torus: if $i\in k^{\times}$, then $T$ is split, and $\gamma_1, \gamma_2, \gamma_3, \gamma_4$ form a basis for $H_1(\Gamma, \Z)^\gk = H_1(\Gamma, \Z)$. If $i\notin k^{\times}$, then $T$ is the product of $\G_m^2$ with the $k(i)$-norm torus. We have $\Fr \gamma_3 = \gamma_4$ and $H_1(\Gamma, \Z)^\gk$ is generated by $\gamma_1, \gamma_2$, and $\gamma_3 + \gamma_4$.

\subsection{Setup for descent}
\label{sec:genus4-setup-descent}

Given a divisor $D \in \dro C$, we wish to determine how to evaluate $\gamma_i(D)$ for each $i$. For each loop, we construct on each component a local parameter which is supported on the nodes. The specific support is determined by the loop. Furthermore, each function must be normalized so as to lie in $\fim$ for the appropriate choice of $\gamma$. An easy, if lengthy, calculation yields Table~\ref{tab:genus4localparam} below.

\renewcommand\arraystretch{2}
\begin{table}[htb]
  \centering
  \[
\begin{array}{cccc}
%  \hline
  \textrm{Loop} & \textrm{Component} & \textrm{Divisor} & \textrm{Local function} \\\toprule
  \gamma_1 & \cxy & [-1:-1:1:1] - [1:1:1:1] & \displaystyle\frac{Z+X}{Z-X} \equiv \frac{X+W}{X-W} \\
  & \czw & [1:1:1:1] - [-1:-1:1:1] & \displaystyle\frac{Z-X}{Z+X} \\ \midrule
  \gamma_2 & \czw & [1:0:0:0] - [0:1:0:0] & \displaystyle\frac{Z-Y}{Z-X} \\
  & \cmzw & [0:1:0:0] - [1:0:0:0] & \displaystyle\frac{Z-X}{Z+Y} \equiv \frac{X+Z}{Y-Z} \\ \midrule
  \gamma_3 & \cxy & [i:i:-1:1] - [1:1:1:1] & \displaystyle\frac{Z-iX}{Z-X}\equiv\frac{X-iW}{X-W} \\
  & \czw & [1:1:1:1] - [0:1:0:0] & \displaystyle\frac{Z-Y}{Z} \\
  & \cmzw & [0:1:0:0] - [i:i:-1:1] & \displaystyle\frac{X}{Z-iX} \\\midrule
  \gamma_3 + \gamma_4 & \cxy & - & \dfrac{Z^2+X^2}{(Z-X)^2} \\
  & \czw & - & \dfrac{(Z-Y)^2}{Z^2} \\
  & \cmzw & - & \dfrac{X^2}{Z^2+X^2} \\\bottomrule
\end{array}
\]

\caption{Local functions for a basis of $H_1(\Gamma, \Z)^\gk$}
\label{tab:genus4localparam}
\end{table}
Each function is written as one on $\Pro^3_k$, which we then restrict to the appropriate component. When two functions are listed, they agree on an open dense set; one can go from one to the other by using the equations for the given component. For example, on $\cxy$ one has $X^2 = ZW$. We have omitted $\gamma_4$; it can be obtained from $\gamma_3$ by applying the substitution $i \mapsto -i$. (Warning: this map is not in general a homomorphism on $k^\times$; for example, let $k=\F_5$, $i=2$, and compare $1+i$ with $1-i$.) The functions for $\gamma_3 + \gamma_4$ are obtained by computing $\Nm (\gamma_3, (t_i))$. As every component in $\gamma_3+\gamma_4$ appears with multiplicity 2, it makes sense that the functions are of degree 2 (when interpreted as maps from the relevant component to $\Pro^1_k$). The divisors for $\gamma_3 + \gamma_4$ are omitted, but may be obtained by taking the norm of the divisors for $\gamma_3$.

 Each function must also be normalized properly. If $i\in k$, then every node is rational over $k$, and the normalization condition is simply that each function is also defined over $k$. Clearly, this is satisfied. If $i \notin k$, then we will use evaluation on the cycles $\gamma_1$, $\gamma_2$, and $\gamma_3 + \gamma_4$. Our construction of the functions for $\gamma_3+\gamma_4$ guarantees that the normalization is correct.

\subsection{Calculation of torsion on Jacobian}
\label{sec:calc-tors-genus4}

As observed in Remark~\ref{rmk:full-torsion}, the prime-to-$p$ rational torsion in the Jacobian $J(K)(p')$ contains a subgroup isomorphic to $T(k)$; we now compute this latter group.

\begin{lemma}\label{lemma:genus4-toric-group}
  Let $i$ be any square root of $-1$ in $\kb$. If $i\in k^\times$, then
  \[
  T(k) \isom \bigg(\frac{\Z}{(q-1)\Z}\bigg)^4.
  \]
  If $i\notin k^\times$, then
  \[
  T(k) \isom \bigg(\frac{\Z}{(q-1)\Z}\bigg)^2 \oplus \frac{\Z}{(q^2-1)\Z}.
  \]
\end{lemma}

\begin{proof}
  As observed in the end of \S \ref{sec:general-facts-1}, if $i\in k^\times$, then $T$ is a split torus, and so $T\isom \G_m^4$; the first claim follows. Note that if we fix an identification of $k^\times$ with $\Z/(q-1)\Z$, then by abuse of notation the isomorphism is given by
\[
\oplus \chi_i: T(k) \to \bigg(\frac{\Z}{(q-1)\Z}\bigg)^4.
\]
We also observed that if $i\notin k^\times$, then $T\isom \G_m^2 \times R_{k(i)/k} \G_m$. As $R_{k(i)/k} \G_m(k) \isom \G_m(k(i)) \isom \Z/(q^2-1)\Z$, the second assertion follows. The isomorphism is given by
\[
(\chi_1,\chi_2,\chi_3): T(k) \to \frac{\Z}{(q-1)\Z} \oplus \frac{\Z}{(q-1)\Z} \oplus \frac{\Z}{(q^2-1)\Z}.
\]
\end{proof}

We now use the descent map to compute the remaining factor in the prime-to-$p$ torsion. Let 
\[
\tau: \dro C \to \Div C_k
\]
be the specialization map, as in \S~\ref{sec:overview}. We define our degree map by
\begin{align*}
  \deg: \Div C_k &\to \Z^3 \\
  D &\mapsto ((D\cdot \cxy), (D\cdot \czw), (D\cdot \cmzw)).
\end{align*}

\begin{lemma}\label{lemma:genus4-component-group}
  The component group of the Jacobian of $C$ over $K$ is
  \[
  \Phi \isom \frac{\Z}{6\Z} \times \frac{\Z}{2\Z}.
  \]
  There are generators $\delta_1, \delta_2$ for the respective cyclic subgroups represented by divisors $D_1, D_2\in \dro C$ respectively such that
  \begin{align*}
    \deg \tau(D_1) &= (0,1,-1)\\
    \deg \tau(D_2) &= (1,-1,2).
  \end{align*}
\end{lemma}

\begin{proof}
  The intersection matrix for $C_k$ is
\renewcommand\arraystretch{1}
  \[
  \left[
    \begin{array}{rrr}
      -4 & 2 & 2 \\
      2 & -4 & 2 \\
      2 & 2 & -4
      \end{array}\right].
    \]
    One then verifies the claims via~\cite[Proposition~8.1.2]{raynaud1970}.
\end{proof}

Consider the pullback of $\sO(1)$ to $C$. Given a linear form $L$ in $X,Y,Z,W$, we write $\dv L$ to mean the divisor of the corresponding section of the pullback sheaf; as observed earlier, this is a section of the canonical bundle of $C$. Similarly, for any function $f$ on $\Pro^3$, $\dv f$ means the divisor of the restriction of $f$ to $C$. 
\begin{lemma}\label{lemma:genus4-torsion-functions}
  We have
  \begin{align*}
    \deg \bigg(\tau \bigg(\dv \frac{Z-W}{Z+W}\bigg)\bigg) &= (0,6,-6)\\
    \deg \bigg(\tau \bigg(\dv \frac{X+Y}{Z+W}\bigg)\bigg) &= (2,2,-4).
  \end{align*}
\end{lemma}

\begin{proof}
  The divisor of each linear section is effective of degree 6. Clearly the specialization of the divisor of $Z-W$ lies entirely on $\czw$, and similarly for $Z+W$ and $\cmzw$. As for $X+Y$, we consider the special fiber $C_k$ as the intersection of $Q_k$ with the degenerate cubic form $\lxy\lzw\lmzw = 0$. Each of $\lxy,\lzw,\lmzw$ gives rise to a type $(1,1)$ divisor on $Q_k$, as does $X+Y=0$. Thus the divisor of $X+Y$ intersects each component of $C_k$ in two points (up to multiplicity). The claim follows.
\end{proof}

 To now calculate the prime-to-$p$ torsion of $J(K)$, we will evaluate the functions in Table~\ref{tab:genus4localparam} on the specializations of the divisors of $X+Y$, $Z-W$, and $Z+W$.

 \begin{lemma}
   The specialization of the divisor $\dv (X+Y)$ on $C_k$ is
   \begin{align*}
      [0:0:1:0] &+ [0:0:0:1] \\
     +  [i:-i:1:1] &+ [-i:i:1:1] \\
     +  [-1:1:-1:1] &+ [1:-1:-1:1].
   \end{align*}
   The first pair lies on $\cxy$, the second on $\czw$, and the third on $\cmzw$.
 \end{lemma}

 \begin{proof}
   As mentioned earlier, the hyperplane $X+Y=0$ in $\Pro^3_k$ intersects every component transversely in two points, and so the calculation of $\dv (X+Y)$ is straightforward. The second claim is easily verified.
 \end{proof}

Let
\begin{align*}
  \hzw(x) &= x^3\bar{\eps}\bigg(x,\frac1x, 1, 1\bigg) \\
  \hmzw(x) &= x^3\bar{\eps}\bigg(x, -\frac1x, -1, 1\bigg)
\end{align*}
Let $\alpha_1, \cdots, \alpha_6$ be the roots of $\hzw(x)$ counted with multiplicity, and similarly let $\beta_1, \cdots, \beta_6$ be the roots of $\hmzw(x)$.
\begin{lemma}
  The specialization of the divisor of $Z-W$ is
  \[
  \sum \bigg[\alpha_i: \frac1\alpha_i:1:1\bigg].
  \]
  The specialization of the divisor of $Z+W$ is
  \[
  \sum \bigg[\beta_i: -\frac1\beta_i:-1:1\bigg].
  \]
\end{lemma}

\begin{proof}
  We first consider $Z-W$. Generically, we wish to solve the system
  \[
  XY=ZW, \quad Z=W, \quad \eps(X,Y,Z,W) = 0.
  \]
  Lemma~\ref{lemma:eps-regular} implies that $\eps = aX^3 + bY^3 + \cdots$, where $a,b \in \ok$ are units. If $Z=W=0$, this implies that $X=Y=0$, which does not occur in $\Pro^3$. Therefore we may assume that $Z=W=1$. Together with the fact that the coefficient $\bar{b} \neq 0$, the claim for $Z-W$ now follows.

  Similar reasoning holds for the divisor of $Z+W$.
\end{proof}

\begin{lemma}\label{lemma:prod-alpha-beta}
  For $c\neq 0$, the following equalities hold:
  \begin{gather*}
    \prod (c-\alpha_i) = \prod (\alpha_i - c) = c^3\frac{\bar{\eps}\big(c, \frac1c, 1,1\big)}{\bar{\eps}(1,0,0,0)} \\
    \prod (c-\beta_i) =  \prod (\beta_i - c) = c^3\frac{\bar{\eps}\big(c, -\frac1c, -1, 1\big)}{\bar{\eps}(1,0,0,0)} \\
    \prod \alpha_i = -\prod \beta_i = \frac{\bar{\eps}(0,1,0,0)}{\bar{\eps}(1,0,0,0)}
  \end{gather*}
\end{lemma}

\begin{proof}
  Observe that the lead coefficients of $\hzw$ and of $\hmzw$ are the coefficient of $X^3$ in $\bar{\eps}$, which equals $\bar{\eps}(1,0,0,0)$. From this, the first two claims are trivial. To evaluate $\prod \alpha_i$, we see that the product equals $\hzw(0)$ divided by its lead coefficient, which is the constant term of $x^3\bar{\eps}(x,\frac1x, 1, 1)$ divided by its lead coefficient. Since $\bar{\eps}$ is a homogeneous cubic, the constant term is the coefficient of $Y^3$, which equals $\bar{\eps}(0,1,0,0)$. The argument for $\prod \beta_i$ is similar.
\end{proof}

We now evaluate the $\gamma_i$ on the divisors discussed above, as well as on certain linear combinations of these divisors. For convenience, we only summarize the results in Table~\ref{tab:genus4divisors} below---the calculations are routine. The only two ``tricks'' are liberal use of Lemma~\ref{lemma:prod-alpha-beta} and that, for each point in the support of the given divisor, one should choose the local parameter in Table~\ref{tab:genus4localparam} which is regular and nonvanishing at that point. (The Table~\ref{tab:genus4localparam} is constructed so that there is always such a choice for the divisors below.) In Table~\ref{tab:genus4divisors}, we write $\eps$ in place of $\bar{\eps}$ for convenience. In the interests of space, we omit $\gamma_4(D)$, which can be obtained by replacing $i$ with $-i$ in $\gamma_3(D)$.

\renewcommand\arraystretch{2}
\begin{table}[htb]
  \centering
  \[
\begin{array}[c]{ccccc}
D  & \gamma_1(D) & \gamma_2(D) & \gamma_3(D) & \gamma_4(D) \\\toprule
% BEGIN RECEIVE ORGTBL hyperplane
\dv X+Y & -1 & -1 & -i & - \\
\dv Z-W & \dfrac{\eps(1,1,1,1)}{\eps(-1,-1,1,1)} & \dfrac{\eps(1,0,0,0)}{\eps(0,1,0,0)} & \dfrac{\eps(1,1,1,1)}{\eps(0,1,0,0)} & - \\
\dv Z+W & 1 & -\dfrac{\eps(0,1,0,0)}{\eps(1,0,0,0)} & i\dfrac{\eps(0,1,0,0)}{\eps(i,i,-1,1)} & - \\ \midrule
\dv \frac{Z-W}{Z+W} & \dfrac{\eps(1,1,1,1)}{\eps(-1,-1,1,1)} & -\dfrac{\eps(1,0,0,0)^2}{\eps(0,1,0,0)^2} & -i\dfrac{\eps(1,1,1,1)\eps(i,i,-1,1)}{\eps(0,1,0,0)^2} & - \\
\dv \frac{X+Y}{Z+W} & -1 & \dfrac{\eps(1,0,0,0)}{\eps(0,1,0,0)} & -\dfrac{\eps(i,i,-1,1)}{\eps(0,1,0,0)} & - \\
\dv \frac{Z-W}{X+Y} & -\dfrac{\eps(1,1,1,1)}{\eps(-1,-1,1,1)} & -\dfrac{\eps(1,0,0,0)}{\eps(0,1,0,0)} & i\dfrac{\eps(1,1,1,1)}{\eps(0,1,0,0)} & - \\ \midrule
\dv \frac{Z^2-W^2}{X+Y} & -\dfrac{\eps(1,1,1,1)}{\eps(-1,-1,1,1)} & 1 & -\dfrac{\eps(1,1,1,1)}{\eps(i,i,-1,1)} & - \\
\dv \frac{(Z+W)^2}{Z-W} & \dfrac{\eps(-1,-1,1,1)}{\eps(1,1,1,1)} & \dfrac{\eps(0,1,0,0)^3}{\eps(1,0,0,0)^3} & -\dfrac{\eps(0,1,0,0)^3}{\eps(i,i,-1,1)^2\eps(1,1,1,1)} & - \\
% END RECEIVE ORGTBL hyperplane
\bottomrule
\end{array}
\]
\caption{Evaluation of loops on certain divisors}
\label{tab:genus4divisors}
\end{table}

\begin{theorem}\label{thm:genus4-torsion}
  Let $K$ be a local field with uniformizer $\pi$ and residue field $k$ of characteristic $p\geq 5$ and order $q$. Let $i \in \kb^\times$ be a fixed square root of $-1$. Let $C$ be the locus in $\Pro^3_K$ given by $XY=ZW$ and
  \[
  (X-Y)(Z-W)(Z+W) = \pi \eps
  \]
  with $\eps\in \ok[X,Y,Z,W]$ a homogeneous cubic satisfying Lemma~\ref{lemma:eps-regular}. Let $J$ be its Jacobian, $J(K)(p')$ the largest torsion subgroup of $J(K)$ with order not divisible by $p$, and $T$ the toric part of the reduction of a N\'eron model for $J$ over $\ok$. Then $J(K)(p')$ lies in a short exact sequence
  \[
  0 \to T(k) \to J(K)(p') \to \frac{\Z}{6\Z} \oplus \frac{\Z}{2\Z} \to 0.
  \]
  More precisely, we have the following.
  \begin{enumerate}
      \item Suppose $i \in k^\times$. Then $J(K)(p')$ is isomorphic to
    \[
    \frac{\Z}{a_0\Z} \oplus \frac{\Z}{b_0\Z} \oplus \frac{\Z}{a_1(q-1)\Z} \oplus \frac{\Z}{b_1(q-1)\Z} \oplus \bigg(\frac{\Z}{(q-1)\Z}\bigg)^2
    \]
    where $a_0a_1 = 6$, $b_0b_1 = 2$, and these constants are computed as follows:

Let $H_1 \subset k^\times$ be the subgroup generated by the entries in the row $\dv \frac{Z-W}{Z+W}$ of Table~\ref{tab:genus4divisors}, let $H_2$ be the group generated by the entries in the row $\dv \frac{X+Y}{Z+W}$, and let $H_3$ be the group generated by the entries in the row $\dv \frac{Z-W}{X+Y}$. (Recall that the last entry $\gamma_4(D)$ is obtained from $\gamma_3(D)$ by replacing $i$ with $-i$.) 

Then $a_1$ is the order of $\frac{H_1\cdot k^{\times 6}}{k^{\times 6}}$ and $a_0=6/a_1$. 

If $H_3 \subset k^{\times 2}$, set $b_1=1$ and $b_0 = 2$. Otherwise, $b_1$ is the order of $\frac{H_2\cdot k^{\times 2}}{k^{\times 2}}$ and $b_0 = 2/b_1$. 

      \item Suppose $i \notin k^\times$. Then $J(K)(p')$ is isomorphic to
    \[
    \frac{\Z}{a_0\Z} \oplus \frac{\Z}{c_0\Z} \oplus \frac{\Z}{a_1(q-1)\Z} \oplus \frac{\Z}{b_1(q-1)\Z} \oplus \frac{\Z}{b_3c_3(q^2-1)\Z}
    \]
    where $a_0a_1 = b_1b_3 = 2$, $c_0c_3 = 3$, and these constants are computed as follows:

    Let $\ell = k(i)$. Let $H_1 \subset k^\times$ be the subgroup generated by the first two entries of row $\dv \frac{Z-W}{Z+W}$ of Table~\ref{tab:genus4divisors} and the norm from $\ell$ to $k$ of the third entry. Let $H_3 \subset k^\times$ be the subgroup generated by the first two entries of row $\dv \frac{Z-W}{X+Y}$ and the norm from $\ell$ to $k$ of the third entry.

    If $3\mid (q-1)$ and $H_1 \subset k^{\times 3}$, or $3\nmid (q-1)$ and
    \[
    \eps(i,i,-1,1)^{(q^2-1)/3} = 1,
    \]
    then $c_0=3$. Otherwise $c_0 = 1$. In either case, $c_3 = 3/c_0$.

    If $\Nm_{\ell/k} \eps(i,i,-1,1) \in k^{\times 2}$, then $b_1 = 2$ and $b_3=1$. Otherwise, $b_1=1$ and $b_3=2$.

    If $H_3 \subset k^{\times 2}$, then $a_0=2$ and $a_1=1$. Otherwise, $a_0=1$ and $a_1=2$.
  \end{enumerate}
\end{theorem}

A word about the notation: letting $\delta_1$ and $\delta_2$ be generators for $\Phi$ as in Lemma~\ref{lemma:genus4-component-group}, the $a_j$ give the contribution of $\delta_1$ to $J(K)(p')$ and the $b_j$ give the contribution of $\delta_2$ in the case $i\in k^\times$. When $i\notin k^\times$, the $a_j$ give the contribution of $3\delta_1$ while the $c_j$ give the contribution of $2\delta_1$. The subscripts on the constants roughly refer to the subscript on the corresponding one-cycle, except the subscript $0$ refers to contributions coming from the kernel of $\nu$ as in Proposition~\ref{prop:computing-nu-with-div-f}. For example, letting $r=2$ in that proposition, if $\gamma_1(3\delta_1)$ is nontrivial, then $a_1=2$. This is not quite correct in all cases; for example, when $i\in k^\times$, we use $a_1=2$ if $\gamma_j(3\delta_1)$ is nontrivial for any $j$. By the symmetry of the factors in $T(k)$ (according to Lemma~\ref{lemma:genus4-toric-group}), it doesn't matter which $\gamma_j$ yields a nontrivial value.

\begin{proof}
  The short exact sequence follows from Proposition~\ref{prop:bosch-liu-phi-constant} and the calculation of the component group $\Phi$ in Lemma~\ref{lemma:genus4-component-group}.

  Suppose $i \in k^\times$. Let $r=2$ in Proposition~\ref{prop:computing-nu-with-div-f}. Note that the nontrivial elements of $\Phi[2]$ are $3\delta_1$, $\delta_2$, and $3\delta_1-\delta_2$, with notation as in Lemma~\ref{lemma:genus4-component-group}. By Lemma~\ref{lemma:genus4-torsion-functions}, the relevant functions in Proposition~\ref{prop:computing-nu-with-div-f} are $\dv \frac{Z-W}{Z+W}$, $\dv \frac{X+Y}{Z+W}$, and $\dv \frac{Z-W}{X+Y}$ respectively. The conditions on the $b_j$ and the 2-part of the $a_j$ now follow from the latter proposition and Corollary~\ref{cor:toric-descent-with-functions}. For example, $\nu(\delta_2) = 0$ if and only if $\gamma_j\big(\dv \frac{X+Y}{Z+W}\big) \equiv 1 \pmod{k^{\times 2}}$ for all $j$, which is equivalent to $[H_2\cdot k^{\times 2}:k^{\times 2}] = 1$. 

  For $r=3$, we wish to determine whether $\nu(2\delta_1)$ is trivial or not; if it is, $3\mid a_0$ and $3\nmid a_1$. If it is not, then $3 \mid a_1$ and $3\nmid a_0$. The determination of $\nu(2\delta_1)$ proceeds in a similar way as in the $r=2$ case above.

  The $i \not\in k^\times$ case is similar, but with the following additional complications. First, suppose $r=2$. Observe that neither $\gamma_2\big(\dv \frac{Z-W}{Z+W}\big)$ nor $\gamma_1\big(\dv \frac{X+Y}{Z+W}\big)$ lie in $k^{\times 2}$ in this case. Next, the factors $\mu(T_j) := \chi_{\gamma_j}(T(k))$ for $j=1,2,3$ are not all isomorphic to each other by Lemma~\ref{lemma:genus4-toric-group}. Let $x\in J(k)$ be any element mapping to $3\delta_1 \in \Phi[2]$. Then even though $\nu(3\delta_1) \neq 0$, the 2-part of the order of $x$ depends on whether $\gamma_3\big(\dv \frac{Z-W}{Z+W}\big)$ lies in $2\mu(T_3)$ or not. Since $\mu(T_3) = \ell^\times$ and the norm gives an isomorphism
  \[
  \frac{\ell^\times}{\ell^{\times 2}} \to \frac{k^\times}{k^{\times 2}}
  \]
  one sees that $\gamma_3\big(\dv \frac{Z-W}{Z+W}\big)$ lies in $2\mu(T_3)$ if and only if $\Nm_{\ell/k} \eps(i,i,-1,1) \in k^{\times 2}$. This explains the condition on the $b_j$; the condition on the $a_j$ follows easily.

  Lastly, let $r=3$. If $3 \nmid (q-1)$, then $\gamma_1$ and $\gamma_2$ automatically take values in $k^{\times 3} = k^\times$. Therefore in order to determine if $\nu(2\delta_1) = 0$, we need to determine if $\gamma_3\big(\dv \frac{Z-W}{Z+W}\big)$ lies in $\ell^{\times 3}$. The condition for the $c_j$ now follows from the observations that $i^3 = -i$ and $k^\times \subset \ell^{\times 3}$.
\end{proof}

\subsection{Rationality of theta characteristics and cube roots of the canonical class}
\label{sec:rati-theta-char-genus4}

\begin{theorem}\label{thm:genus4-thetachar}
  Let $C$ be as in the hypotheses of Theorem~\ref{thm:genus4-torsion}. Let $i \in \kb^\times$ be a fixed square root of $-1$. Define groups $T_j \subset k^\times$ for $j=1,\dots,4$ as follows.
  
  If $i\in k^\times$, let $T_1$ be the subgroup generated by the entries in row $\dv (X+Y)$ of Table~\ref{tab:genus4divisors}. Similarly, let $T_2$, $T_3$, $T_4$ be the subgroups generated by rows $\dv (Z-W)$, $\dv (Z+W)$, and $\dv \frac{Z^2-W^2}{X+Y}$ respectively. 

  If $i\notin k^{\times}$, let $\ell = k(i)$, and let each $T_j$ be computed for the corresponding row as above, but this time taking the subgroup generated by the first two entries of the row, plus the norm from $\ell$ to $k$ of the third entry.

  Then $C$ has a rational theta characteristic if and only if some $T_j$ lies in $k^{\times 2}$.
\end{theorem}

\begin{proof}
  We apply Corollary~\ref{cor:toric-descent-with-functions}. The canonical divisor may be given by $\dv X+Y$, which lies in $\dt C$. There are 4 elements in $\Phi[2]$; the corresponding functions as in Corollary~\ref{cor:toric-descent-with-functions} are $1$, $\frac{Z-W}{X+Y}$, $\frac{Z+W}{X+Y}$, and $\frac{Z^2-W^2}{(X+Y)^2}$. Observe that
\[
\gamma_i\bigg(\dv (X+Y) + \dv\frac{Z-W}{X+Y}\bigg) = \gamma_i (\dv Z-W);
\]
one does similar calculations for the other functions.

  When $i\notin k^\times$, by Corollary~\ref{cor:descent-normal-torus} we need to evaluate $\gamma_1,\gamma_2$, and $\gamma_3+\gamma_4$. This explains why we take the norm of the third entry.
\end{proof}

\begin{theorem}\label{thm:genus4-cube-root}
  Let $C$ be as in Theorem~\ref{thm:genus4-torsion}. If $i\in k^\times$, let $S_1\subset k^\times$ be the subgroup generated by the entries of row $\dv (Z-W)$ in Table~\ref{tab:genus4divisors}. Similarly, let $S_2$, $S_3$ be the subgroups generated by the rows $\dv {Z+W}$ and $\dv \frac{(Z+W)^2}{Z-W}$ respectively. If $i \not\in k^\times$ and $3\mid (q-1)$, let $S_i$ be the subgroups generated by the first two entries of the corresponding rows and the norm from $\ell$ to $k$ of the third entry. In these two cases, $C$ has a rational cube root of the canonical class if and only if some $S_i$ lies in $k^{\times 3}$.

  If $i \not\in k^\times$ and $3\nmid (q-1)$, then $C$ has a rational cube root of the canonical class if and only if $\eps(i,i,-1,1)^{(q^2-1)/3} = 1$.
\end{theorem}

\begin{proof}
  The proof is analogous to that of Theorem~\ref{thm:genus4-thetachar}. Here, we use $\dv (Z-W)$ as our canonical divisor, observing that it lies in $\Div^{\{3\}} C$. We have $\Phi[3]$ has 3 elements with corresponding functions $1$, $\frac{Z+W}{Z-W}$, and $\frac{(Z+W)^2}{(Z-W)^2}$. Thus in Theorem~\ref{thm:descent-with-nu}, we set $r=3$ and use the divisors of the sections in the statement above. The case when $i\in k^{\times}$ follows from Theorem~\ref{thm:descent-with-nu}. The case when $i \not\in k^{\times}$ but $3\mid(q-1)$ follows from Corollary~\ref{cor:descent-normal-torus}.

  Now suppose that $i \not\in k^\times$ and $3\nmid (q-1)$. Since $k^{\times 3} = k^\times$, we automatically have $\gamma_1(D) = \gamma_2(D) = 1$ for $D$ the three relevant divisors. Therefore $C$ has a rational cube root of the canonical class if and only if $\gamma_3(D)$ and $\gamma_4(D)$ lie in $\ell^{\times 3}$. Since these are conjugate, it suffices that $\gamma_3(D) \in \ell^{\times 3}$. The claim now follows from the observations that $k^\times \subset \ell^{\times 3}$ and $i^3 = -i$.
\end{proof}

\begin{example}
  If $K=\Q_p$ with $p\equiv 5 \pmod{12}$, then $i\in \F_p^\times$ and $\frac{\F_p^\times}{\F_p^{\times 3}} = 1$. Therefore $C$ has a rational cube root of the canonical class for every $\eps$ satisfying Lemma~\ref{lemma:eps-regular}.
\end{example}

\section{Closing remarks}
\label{sec:closing-remarks}

For $C$ as in the hypotheses of \S~\ref{sec:eval-divis-1}, I hope to extend the above methods to compute the full rational torsion on the Jacobian of $C$. This requires knowing the torsion in the kernel of reduction; the existence of a $p$-adic uniformization for the Jacobian leads me to believe that such a computation is tractable. However, if $\Phi[p]\neq 0$, one must also determine if the $p$-primary torsion in $J(k)$ lifts to torsion in $J(K)$; this could be resolved if, for example, one knew how to determine if an element in the associated formal group of $J$ is divisible by $p$. Another goal is determining the representation of the absolute Galois group of $K$ on the torsion. In the prime-to-$p$ case, the action of the inertia group on the torsion under our hypotheses is well-known via the Picard-Lefschetz formula. Combining with the descent map described in this paper should be enough to deduce the Galois representation on the torsion.

It is unfortunate that the above methods do not, as is, work for general curves with totally degenerate reduction. One obstruction to generalizing the map in Theorem~\ref{thm:descent-with-nu} is that an algebraic torus need not be principally decomposable. To start, a robust structure theory for the character group of tori over $k$ is needed. However, for most cyclic groups $G$, there are infinitely many isomorphism classes of indecomposable $\Z[G]$-modules~\cite{heller-reiner1962}.

Another problem arises when $\gk$ permutes the vertices of $\Gamma$ in a nontrivial way. In this case, the construction $\Nm \gamma$ gives a way of generalizing the method in \S~\ref{sec:the-descent-map}, but in practice it is not quite satisfactory. For example, if $\alpha \in \ok^\times$ is a nonsquare, then one cannot at present use this method to determine if
\[
C: y^2 = \alpha g^2 + \pi h,
\]
where $g$ and $h$ are as in hypothesis $(H)$, has a rational theta characteristic.

\bibliography{/home/sharif/Research/biblio}
\bibliographystyle{plainurl}
\end{document}